\documentclass{siamart251104}
\usepackage{amsmath}
\usepackage{framed}
\usepackage{float}
\usepackage{diagbox}
\usepackage{bbold}
\usepackage{hyperref}
\usepackage{enumitem}
\usepackage{graphicx} % Required for \includegraphics
\usepackage{caption} % Required for caption customization
\usepackage{subcaption} % Required for subfigures

\newtheorem{thm}{Theorem}

\newtheorem{lem}{Lemma}

\newtheorem{prop}{Proposition}
\newtheorem{rmk}{Remark}

\newcommand{\dt}{\Delta t}

\newcommand{\ybar}{\bar{y}}
\newcommand{\m}[1]{\mathbf{#1}}

\newcommand{\vy}{\m{y}}
\newcommand{\vz}{\m{z}}
\newcommand{\vb}{\m{b}}

\newcommand{\ve}{\m{e}}

\newcommand{\vh}{\m{h}}

\newcommand{\mQ}{\m{P}} %stupid definition because I was lazy

\newcommand{\mI}{\m{I}}
\newcommand{\mA}{\m{A}}
\newcommand{\mB}{\m{B}}

\newcommand{\mAh}{\m{\hat{A}}}

\title{Stable corrections for  perturbed diagonally  implicit Runge--Kutta methods}

\author{John Driscoll\thanks{Mathematics Department, University of Massachusetts Dartmouth,285 Old Westport Rd. North Dartmouth, MA 02747.}\and Sigal Gottlieb*\thanks{sgottlieb@umassd.edu}\and
Zachary J. Grant*\and
César Herrera\thanks{Department of Mathematics, Purdue University, 150 North University Street.
West Lafayette, Indiana 47907}\and
Tej Sai Kakumanu*\and
Michael H. Sawicki*\and
Monica Stephens\thanks{Mathematics Department, Spelman College, 350 Spelman Lane S.W.  Atlanta, GA 30314} 
}

\begin{document}
\maketitle

%\begin{keywords}
%\end{keywords}

\bibliographystyle{siam}

\begin{abstract}  
A mixed accuracy framework for Runge--Kutta methods presented in \cite{Grant2022}
and applied to diagonally implicit Runge--Kutta  (DIRK) methods
can significantly speed up the computation by replacing the 
implicit solver by less expensive low accuracy approaches
such as lower precision computation of the implicit solve, under-resolved iterative solvers,
or simpler, less accurate models for the implicit stages. 
Understanding the effect of the perturbation errors introduced by the 
low accuracy computations enables the design of  stable and accurate
mixed accuracy DIRK methods where the errors from the  
low-accuracy computation  are damped out by multiplication 
by $\dt$ at  multiple points in the simulation, 
resulting in a more accurate simulation than if 
low-accuracy was used for all computation.
To improve upon this,  explicit corrections were proposed and
analyzed for accuracy in \cite{Grant2022}, and their performance was 
tested in \cite{Burnett1, Burnett2}. Explicit corrections work well when the 
time-step is sufficiently small,  but may introduce instabilities 
when the time-step is larger. In this work, the stability of the mixed accuracy approach
is carefully studied, and used to design novel stabilized correction approaches.
  \end{abstract}


\noindent{\bf Classification codes:} 65Mxx, 65M20, 	65L04, 65M70, 65L05.

\section{Overview} 
Diagonally implicit Runge--Kutta (DIRK) methods \cite{CarpenterRK} are often
used for the time evolution of a system of ordinary
differential equations (ODEs) of the form
\begin{eqnarray} \label{ODE}
    y' &=& f(y), \; \; \; y(0) = y_0.
\end{eqnarray} 
Such systems may result from the semi-discretization of a partial differential equation (PDE).
DIRK methods require costly implicit solves, but
their large linear stability  regions allow for 
larger step-sizes. This becomes an important consideration when the problem is {\em stiff}.  
In such cases, explicit methods are not feasible because the time-step is severely limited by stability rather than accuracy considerations.

A mixed accuracy framework for DIRK methods
allows us to speed up the computation of the implicit solves without degrading the overall accuracy \cite{Grant2022}. Such approaches may include lower precision computation of the implicit solve, under-resolved iterative solvers,
or simpler, less accurate models for the implicit stages. 
The key idea is that the expensive part of 
the implicit solve  can be evaluated using a 
computationally inexpensive strategy.  

Using the theory in \cite{Grant2022} we can understand 
the effect of the perturbation errors introduced by the 
low accuracy computations. This allows the 
design of DIRK methods that mitigate the impact of 
the low accuracy perturbations on the overall solution. 
The DIRK methods can be designed so that the errors from the  
low-accuracy computation  are damped out by multiplication 
by $\dt$ at  multiple points in the simulation, 
resulting in a more accurate simulation than if 
low-accuracy was used for all computation.
However, the resulting methods are only first order for sufficiently small time-steps.
In \cite{Grant2022}, explicit corrections were proposed to improve the 
accuracy of the mixed precision solutions. The performance and stability 
of this  approach in the mixed precision case was 
tested in \cite{Burnett1, Burnett2}, and was shown to 
work as predicted.  However, it was shown that the explicit
corrections may introduce instabilities.
In this work, we aim to better understand the impact
of low-accuracy perturbations on the stability of the
approach presented in \cite{Grant2022}.
In particular, we will use this understanding 
to better design the low accuracy approaches for 
the implicit solves, and to  design  stabilized
correction approaches.

The paper is organized as follows: In Section \ref{sec:stability}
we present the accuracy and stability analysis of perturbed methods.
In Section \ref{sec:sec3num} we use a nonlinear model to show the impact
of linearization and of mixed precision, and verify that this matches
with the theory in Section \ref{sec:stability}.
In Section \ref{sec:corr} we introduce our approach to corrections that
enhance stability and accuracy. These depend on a stabilization matrix $\Phi$, and approaches
to defining such matrices are describe in Section \ref{sec:phi}.
In Section \ref{sec:NumCor} we study numerically the impact of the stabilized
correction approaches on three model problems. Finally, in Section \ref{sec:conclusions}
we summarize our conclusions for this work.

\section{Accuracy and stability analysis of perturbed methods} \label{sec:stability}

We begin with an initial value problem of the form
\eqref{ODE}, where the function $f$ is contractive \cite{hairer1996solving2}:
\begin{subequations}
\begin{eqnarray} \label{contractivity}
    \left( x -  y , f(x) -  f(y) \right)  \leq 0 \; \; \; 
    \mbox{for any $x, y$},
\end{eqnarray}
and its derivative is bounded
\begin{eqnarray} \label{Derivative}
\left\| f'(y) \right\| \leq L \; \; \; \mbox{for some $L>0$.}
\end{eqnarray}
\end{subequations}
We focus particularly on the case where $f$ is contractive, 
because we do not expect a  non-contractive process to effectively
damp out   the perturbations introduced by the mixed accuracy approach.

We evolve the solution forward using   a DIRK method 
\begin{subequations} \label{DIRK}
\begin{eqnarray} 
    z^{(i)} & = & z_n + \dt \sum_{j=1}^{i} a_{ij} f(z^{(j)})   \\
    z_{n+1} & = & z_n + \dt \sum_{i=1}^s b_{i} f(z^{(i)}) .
\end{eqnarray}
\end{subequations}
We assume that this method has 
coefficients given in an $s \times s$
lower triangular matrix $\mA = (a_{ij})$, 
and the column vector $\vb = (b_i)$ (and the associated matrix $\mB = diag(\vb)$)
such that 
\begin{subequations} \label{Coefficients}
\begin{eqnarray} 
    a_{ii} & \geq & 0, \; \; \; b_i  \geq  0 ,\; \; \; 
    c_i =   \sum_j a_{ij} \; \; \; \mbox{are distinct} \\
    M & = &  \mB \mA + \mA^T \mB - \vb \vb^T 
    \; \; \; \mbox{is semi positive definite} .
\end{eqnarray} 
\end{subequations}
Note that these conditions  mean that the method 
satisfies the conditions for a type of nonlinear inner product 
stability known as B-stability \cite{CrouzeixBstab}.

To make the implicit stages cheaper to invert,
we  chose to replace $f(y^{(i)})$ with another function $f_{\varepsilon}(y^{(i)})$   for computing the stage values $ y^{(i)}$.
The method
then takes the form:
\begin{subequations} \label{MPDIRK}
\begin{eqnarray} 
    y^{(i)} & = & y_n + 
    \dt \left( \sum_{j=1}^{i-1} a_{ij} f(y^{(j)})
    + a_{ii} f_{\varepsilon}(y^{(i)}) \right) \\
    y_{n+1} & = & y_n + \dt \sum_{i=1}^s b_{i} f(y^{(i)}) .
\end{eqnarray}
\end{subequations}
This strategy introduces a perturbation
\[ h(y) = f(y) - f_{\varepsilon}(y) \]
into the internal stages, which can then be expressed as 
\begin{eqnarray} 
    y^{(i)}  & = & y_n + \dt \sum_{j=1}^{i} a_{ij} f(y^{(j)})
    - \dt a_{ii} h(y^{(i)}) .
\end{eqnarray}
This is a common approach that is used whenever the implicit 
stage is not evaluated exactly, e.g. 
when $f$ is approximated by a lower precision computation or a linear operator.
 In fact, a perturbation of this form is introduced  whenever the implicit stage 
 is approximated by  some iterative procedure such as  Newton's iteration. Of particular interest to us are
 nonsmooth perturbations that stem from the use
 of mixed precision computations. In the case where $f$
 is computed in high precision, and $f_\varepsilon$
 is computed in low precision, the resulting
 $h = f - f_\varepsilon$ is not a continuous function.

In this section we bound the growth of the perturbation
errors 
by studying the difference 
between  \eqref{DIRK} and \eqref{MPDIRK}: 
\begin{subequations} \label{errors}
\begin{eqnarray} 
 z^{(i)} -  y^{(i)} & = & z_n - y_n + \dt 
 \sum_{j=1}^{i} a_{ij}  \left( f(z^{(j)})  - f(y^{(j)}) \right) 
    + \dt a_{ii} h(y^{(i)})  \\
 z_{n+1} -   y_{n+1} & = & z_n - y_n + 
 \dt \sum_{i=1}^s b_{i} \left( f(z^{(i)})  - f(y^{(i)}) \right) .
\end{eqnarray}
\end{subequations}
To simplify the notation, we temporarily pretend that 
$y$ and $z$ are scalars,  to avoid 
the use of cumbersome Kronecker products. 
However, everything in this work carries through to the 
trivially (but with some painful notation) to the vector case.

The following lemma  bounds  the growth of the perturbation errors from timestep to 
timestep using the errors from the internal stages. This growth will depend on the size 
of the perturbation and the stiffness  of the problem.

\begin{lem} \label{lem1}
    Given a differential equation of the form \eqref{ODE}
    that is evolved forward with the 
method \eqref{MPDIRK} using the 
function $f_{\varepsilon}$  where
    \[ \left\| h(y^{(i)})\right\| = \left\|f(y^{(i)}) - f_{\varepsilon}(y^{(i)}) \right\| \leq \varepsilon_i. \]
If the coefficients of \eqref{MPDIRK}
satisfy the conditions \eqref{Coefficients},
then the growth of the errors resulting from  $h$
is bounded by:
%   \begin{eqnarray}\label{boundMPgrowth}
% \|y_{n+1} - z_{n+1}\|^2 &\leq & \|y_n - z_n\|^2 
% + 2 \dt  \mu  \sum_{i=1}^s b_{i} 
% \left\| z^{(i)} -  y^{(i)} \right\|^2 \\
% &+& 2 \dt^2 L \varepsilon 
% \sum_{i=1}^s b_{i} a_{ii}
% \left\| z^{(i)} -  y^{(i)} \right\| \;  . \nonumber
% \end{eqnarray}  
  \begin{eqnarray}\label{boundMPcont}
\left\|z_{n+1} - y_{n+1}\right\|^2 &\leq & 
\left\|z_n - y_n\right\|^2  
+ 2 \dt^2 L  \sum_{i=1}^s \varepsilon_i b_{i} a_{ii} \left\| z^{(i)} -  y^{(i)} \right\|  .
    \end{eqnarray}
\end{lem}

\begin{proof}
We look at the inner product of the difference between $z$ and $y$ at each time-step, where for simplicity we define
$\psi_i = \dt ( f(z^{(i)}) - f(y^{(i)})) $, and the associated vector $\Psi$.
\begin{eqnarray*}
\|z_{n+1} - y_{n+1}\|^2   &=&  \|z_n - y_n +
 \sum_{i=1}^s b_{i}  \psi_i  \|^2 \\
 &=& \| z_n - y_n\|^2 +
2  \sum_{i=1}^s b_{i} (z_n - y_n)^T \psi_i
 + \left( \vb \Psi ,\vb \Psi \right) \\
  &=&  \| z_n - y_n\|^2 
 + \left( \vb \Psi ,\vb \Psi \right)  \\ 
 & + & 2 \sum_{i=1}^s b_{i} \psi_i^T \left(
    ( z^{(i)} -  y^{(i)} ) -  \sum_{j=1}^{i} a_{ij} \psi_j 
    - \dt a_{ii} h(y^{(i)})
    \right) \\
    &=&  \| z_n - y_n\|^2  - \left(  \Psi ,M \Psi \right) 
+ 2 \sum_{i=1}^s b_{i} 
 \left( \psi_i, z^{(i)} -  y^{(i)}  - \dt a_{ii} h(y^{(i)}) \right)  \\
& \leq & \| z_n - y_n\|^2 + 2 \dt \sum_{i=1}^s b_{i} 
 \left( f(z^{(i)}) -  f(y^{(i)}) , z^{(i)} -  y^{(i)} 
 - \dt a_{ii}  h(y^{(i)}) \right) ,
\end{eqnarray*}
the inequality follows from the fact that   $M$ is semi-positive definite by assumption, 
so that $( \Psi, M \Psi)  \geq 0$.
  Using the contractivity of $f$,
 and 
\[ \left\| f(z) -  f(y) \right\| 
= \left\| f'(\xi) \right\| \left\| z -  y \right\|  \leq L \left\|z -  y \right\| ,\] we have 
 \begin{eqnarray*}
\left\| z_{n+1} - y_{n+1} \right\|^2  
& \leq &
\left\| z_n - y_n\right\|^2   + 
 2 \dt^2 L \sum_{i=1}^s b_{i} a_{ii} 
 \left\|z^{(i)} -  y^{(i)} \right\|
\left\| h(y^{(i)})\right\| ,
 \end{eqnarray*}
 using the bound on $\left\| h(y^{(i)})\right\|$
we obtain our result.
{\em Note that this proof approach follows directly from
 \cite{CrouzeixBstab}.}
\end{proof}

% \begin{rmk}
%     If the method is not B-stable, we expect some 
%     extra growth at the final stage
% \begin{eqnarray*}
% \left|z_{n+1} - y_{n+1}\right|   
% &\leq &  \left|z_n - y_n \right|+
% \dt  \sum_{i=1}^s b_{i} \left| f(z^{i}) -   f(y^{i})\right| \\
% &\leq &  \left|z_n - y_n \right|+
% \dt L  \sum_{i=1}^s b_{i} \left| z^{i} - y^{i}\right|
% \end{eqnarray*}
% \end{rmk}
%%%%%%%%%%%%%%%%%%%%%%%%%%%%%%%%%%%%%%%%%%%%%%%%%%%%%
%%%%%               Monica Comment               %%%%
%  This would have helped me if it was located directly following Theorem 1, but I understand you are using it as a lead-in into the next section.
%%%%%%%%%%%%%%%%%%%%%%%%%%%%%%%%%%%%%%%%%%%%%%%%%%%%%
{Lemma 1 expresses that the growth of the errors
depends on the size of the perturbation at each stage $h(y^{(i)})$, 
the stiffness of the problem as represented by $L$, 
and the internal stage errors $ \left\| z^{(i)} - y^{(i)}\right\|$.}
In the next section we bound the internal stage errors 
resulting from the perturbation, and this enables us to
bound the final time error more directly.

\subsection{Bounding the internal stage errors}
Our goal in this section is to bound the internal stage
perturbation errors 
$\| z^{(i)} - y^{(i)} \|$ 
to better understand the resulting error
at each time-step:
\begin{eqnarray*}
\| z_{n+1} - y_{n+1}\|^2 &\leq &
\|z_n - y_n \|^2   + 2 \dt^2 L 
\sum_{i=1}^s b_{i} a_{ii}
\underbrace{\left\| z^{(i)} -  y^{(i)} \right\|}_{\mbox{stage}}   
\underbrace{\left\| h(y^{(i)}) \right\|}_{\mbox{perturbation}}.
    \end{eqnarray*}

 The next subsections gradually build this theory
for the one stage  SDIRK2 (also known as the implicit midpoint rule IMR), 
the two stage SDIRK3 method, 
and finally a general $s$-stage DIRK method.

\subsubsection{Implicit midpoint rule}
The second order  SDIRK2 or implicit midpoint rule (IMR)  
can be written in its  Runge--Kutta form
\begin{subequations} \label{IMR}
\begin{eqnarray} 
z^{(1)} & = & z_n + \frac{1}{2} \dt f(z^{(1)})   \\ 
z_{n+1} & = & z_n + \dt f(z^{(1)}) .
\end{eqnarray}
\end{subequations}
The mixed precision version of this method is 
\begin{subequations} \label{MPIMR}
\begin{eqnarray} 
y^{(1)} & = & y_n + \frac{1}{2} \dt f_\varepsilon (y^{(1)})   \\ 
y_{n+1} & = & y_n + \dt f(y^{(1)}) .
\end{eqnarray}
\end{subequations}

To bound the first stage we start with the following proposition:

\begin{prop} \label{prop1}
Given a contractive function $f$ and a constant $\delta \geq 0 $
\begin{eqnarray} \label{ImplicitStageBound}
\left\|z - y \right\|
& \leq &
\left\|(z - y)   - \delta  \left(f(z) -  f(y)  \right) \right\|. 
\end{eqnarray}
for any   $y$ and $z$. 
\end{prop}
\begin{proof}
Begin by noting that the contractivity of $f$ gives
\begin{eqnarray*}
    \left(z - y, f(z) -  f(y) \right) \leq 0 
\end{eqnarray*}  
so that
\begin{eqnarray*} 
\left\|z - y \right\|^2
& \leq &
\left\|z - y \right\|^2  
- 2 \delta   \left(z - y, f(z) -  f(y) \right)  + 
\delta^2 \left\|f(z) - f(y) \right\|^2 \\
& = & \| (z - y)  - \delta \left( f(z) - f(y)\right) \|^2.
\end{eqnarray*}

\end{proof}

A consequence of Proposition \ref{prop1} 
is that for the implicit midpoint rule we have
\begin{eqnarray*}
\left\| z^{(1)} -  y^{(1)} \right\|
& \leq & 
\left\|z^{(1)} -  y^{(1)} - \frac{1}{2} \dt \left(f(z^{(1)}) -  f(y^{(1)}) \right) \right\| \\
%%%%%%%%%%%%%%%%%%%%%%%%%%%%%%%%%%%%%%%%%%%%%%%%%%%%%
%%%%%               Monica Comment               %%%%
%       horizontal space between \delt t \hspace{0.02in} h(y(i))
%%%%%%%%%%%%%%%%%%%%%%%%%%%%%%%%%%%%%%%%%%%%%%%%%%%%%
& = & \left\| z_n - y_n  + \frac{1}{2} \dt  \; h(y^{(1)}) \right\| \leq  
 \left\| z_n - y_n \right\| + \frac{1}{2} \dt \varepsilon_1.
\end{eqnarray*}
Plugging this into the error bound \eqref{boundMPcont} we get:
\begin{eqnarray*}
\|z_{n+1} - y_{n+1}\|^2 &\leq & \|z_n - y_n\|^2  
+  \dt^2 L  \left\| z^{(1)} -  y^{(1)} \right\| \;  \varepsilon_1 \\
&\leq &  \|z_n - y_n\|^2  +  \dt^2 L \left(
        \left\|z_{n} - y_{n}\right\|   +  \frac{1}{2} \dt  \varepsilon_1  
        \right) \varepsilon_1   \\
 & \leq  & \left( {\|z_n - y_n\|} + \frac{1}{2} \dt^2 L \varepsilon_1
 \right)^2 + \left( \frac{1}{2}   - \frac{1}{4} \dt L \right)  \dt^3 L \varepsilon_1^2 
    \end{eqnarray*}
so that
\begin{subequations}
\begin{eqnarray}
    \|z_{n+1} - y_{n+1}\| &\leq&
    \|z_n - y_n\| + \frac{1}{2} \varepsilon_1 \dt^2 L  
+  \varepsilon_1 \dt \sqrt{\frac{\dt L }{2}}  .
\end{eqnarray}
Additionally, if  $ \dt \geq \frac{2}{L}$ we can  conclude that:
\begin{eqnarray}
%%%%%%%%%%%%%%%%%%%%%%%%%%%%%%%%%%%%%%%%%%%%%%%%%%%%%
%%%%%               Monica Comment               %%%%
%     \\z_n - y_n\\ for consistency
%%%%%%%%%%%%%%%%%%%%%%%%%%%%%%%%%%%%%%%%%%%%%%%%%%%%%
 \|z_{n+1} - y_{n+1}\| \leq {\|z_n - y_n\| }
+ \frac{1}{2} \varepsilon_1 \dt^2 L .
\end{eqnarray}
\end{subequations}
This is a reasonable assumption since we 
are dealing with stiff problems where $L$ is large,
which is the scenario that DIRK methods are intended to handle.

\subsubsection{The two stage third order SDIRK method}
When dealing with two stages, 
the analysis becomes more involved.
The SDIRK3 method \cite{norsett} is given by 
\begin{subequations}
\begin{eqnarray}\label{SDIRK3} 
z^{(1)} & = & z_n + \gamma\dt f(z^{(1)})   \\ 
    z^{(2)} & = & z_n + (1-2\gamma)\dt  f(z^{(1)}) + \gamma\dt f(z^{(2)})\\
    z_{n+1} & = & z_n + \frac{\dt}{2}f(z^{(1)}) + \frac{\dt}{2}f(z^{(2)}),
\end{eqnarray}
\end{subequations}
where $\gamma = \frac{\sqrt{3} + 3}{6}$.
We can verify that conditions \eqref{Coefficients} are satisfied.
The mixed precision version of this method is
\begin{subequations} \label{MPSDIRK3} 
\begin{eqnarray}
y^{(1)} & = & y_n + \gamma\dt f_\varepsilon (y^{(1)})   \\ 
    y^{(2)} & = & y_n + (1-2\gamma)\dt  f(y^{(1)}) + \gamma\dt f_\varepsilon (y^{(2)})\\
    y_{n+1} & = & y_n + \frac{\dt}{2}f(y^{(1)}) + \frac{\dt}{2}f(y^{(2)}),
\end{eqnarray}
\end{subequations}

From Proposition \ref{prop1} we know that the first stage errors
are bounded by
\begin{eqnarray*}
\left\| z^{(1)} - y^{(1)} \right\| & \leq &
\left\|z_{n} - y_{n} \right\|   
+ a_{11} \varepsilon_1 \dt   .
\end{eqnarray*}
We now proceed to the second stage, once again using 
Proposition \ref{prop1}:
\begin{eqnarray*}
\left \| z^{(2)} - y^{(2)}\right\| & \leq &
\left \| z^{(2)} - y^{(2)}  - \dt   a_{22}  \left(f(z^{(2)}) -  f(y^{(2)}) \right) \right\| \\
& = & \left \|z_n - y_n + \dt a_{21} \left( f(z^{(1)}) - f(y^{(1)}) \right) + \dt a_{22} h(y^{(2)}) \right\| \\
& = & \left \|z_n - y_n + 
\frac{a_{21}}{a_{11}} \left( (z^{(1)} - y^{(1)})
- (z_n - y_n) - \dt a_{11} h(y^{(1)}) \right) \right. \\
&& \left. + \dt a_{22} h(y^{(2)}) \right\| \\
& \leq  & \left(1 +  \frac{|a_{21}|}{a_{11}}\right) 
\left\|z_n - y_n\right\| +
\frac{|a_{21}|}{a_{11}}  \left\| z^{(1)} - y^{(1)} \right\| 
+  \dt \left(|a_{21}|\varepsilon_1 + a_{22} \varepsilon_2\right)  \\
& \leq & \left(1 +  \frac{2 |a_{21}|}{a_{11}}\right) 
\left\|z_n - y_n\right\| +
 \dt \left(2 |a_{21}| \varepsilon_1+ a_{22} \varepsilon_2\right) .
\end{eqnarray*}
% Letting $\varepsilon = max\left\{\varepsilon_1, \varepsilon_2\right\}$,  the error growth at the 
% second stage of the SDIRK method \eqref{SDIRK3} is bounded by
% \begin{eqnarray*}
% \left \| z^{(2)} - y^{(2)}\right\| & \leq &
% \left( \frac{5 \gamma - 2}{\gamma} \right) 
% \left\|z_n - y_n\right\| +
% \varepsilon \dt \left(5 \gamma - 2 \right) .
% \end{eqnarray*}
%%%%%%%%%%%%%%%%%%%%%%%%%%%%%%%%%%%%%%%%%%%%%%%%%%%%%
%%%%%               Monica Comment               %%%%
%    Changing to this will place the subscript under \max inline \epsilon = \max\limits_i \epsilon_i  
%  Is it premature to define \epsilon here?  You do it later.
%%%%%%%%%%%%%%%%%%%%%%%%%%%%%%%%%%%%%%%%%%%%%%%%%%%%%
Plugging this back into Lemma \ref{lem1}, and using the coefficients
of the scheme we obtain
\begin{eqnarray*}
    \|z_{n+1} - y_{n+1}\|^2 &\leq &
        \|z_n - y_n\|^2  + 2 \dt^2 L \sum_{i=1}^s b_{i} a_{ii}
        \left\| z^{(i)} -  y^{(i)} \right\| \;  \left\| h(y^{(i)}) \right\| \\
  &\leq &  \|z_n - y_n\|^2  +    \gamma \dt^2 L \varepsilon_1 
         \left\| z^{(1)} -  y^{(1)} \right\| 
        +  \gamma \dt^2 L \varepsilon_2  \left\| z^{(2)} -  y^{(2)}\right\|    \\
  &\leq &    \|z_n - y_n\|^2  +    
  \gamma \dt^2 L \varepsilon_1 
        \left( \left\|z_{n} - y_{n} \right\|   +   \gamma  \varepsilon_1 \dt \right) \\
& + & \gamma \dt^2 L \varepsilon_2 \left(  
\left( 1 + 2 \frac{|1- 2\gamma|}{\gamma} \right) 
\left\|z_n - y_n\right\| +
 \dt \left(2 |1- 2 \gamma| \varepsilon_1 + \gamma \varepsilon_2 
        \right) \right) \\
  &= &    \|z_n - y_n\|^2  +    
  \gamma \dt^2 L \varepsilon_1 
        \left( \left\|z_{n} - y_{n} \right\|   +   \gamma  \varepsilon_1 \dt \right) \\
& + & \dt^2 L \varepsilon_2 \big( 
\left( 5 \gamma -2  \right) 
\left\|z_n - y_n\right\| +
\gamma \dt \left( (4\gamma -2) \varepsilon_1 + \gamma \varepsilon_2 \right)  \big) .
\end{eqnarray*}
For simplicity, we let $\varepsilon = \max_i \varepsilon_i$, and get
\begin{eqnarray*}
    \|z_{n+1} - y_{n+1}\|^2 &\leq &
     \|z_n - y_n\|^2  +  
    2 \dt^2 L \varepsilon \left( 3 \gamma -1 
    \right) \left\|z_{n} - y_{n} \right\| 
+   2 \dt^3 L \varepsilon^2 \gamma \left(3 \gamma -1
           \right).
\end{eqnarray*}
Completing the square  we get
\begin{eqnarray*}
    \|z_{n+1} - y_{n+1}\|^2 &\leq &
\big( \|z_n - y_n\|  +   \dt^2 L \varepsilon 
\left( 3 \gamma -1 \right)\big)^2
+ \dt^3 L \varepsilon^2 \left( 3 \gamma -1 \right)
\big( 2 \gamma - \dt L \left( 3 \gamma -1 \right)
\big). %\\
% &\leq &
% \big( \|y_n - z_n\|  +   \dt^3 L \varepsilon 
% \left( 3 \gamma -1 \right)\big)^2
% + \dt^3 L \varepsilon^2 \left( 6 \gamma^2 -2 \gamma \right).
   \end{eqnarray*}  
Neglecting the $O(\dt^4)$ term, which is negative, we get
\begin{subequations}
\begin{eqnarray} \label{SDIRKbound1}
    \|z_{n+1} - y_{n+1}\| &\leq &
    \|z_n - y_n\|  + 
 \varepsilon \dt^2 L  \left(3 \gamma - 1\right)
  + \varepsilon  \dt \sqrt{\dt L } \sqrt{6 \gamma^2 - 2 \gamma}.
\end{eqnarray}
Alternatively, if we wish to consider only 
$\dt \geq \frac{2 \gamma}{(3 \gamma - 1)L}$,
then we have 
\begin{eqnarray} \label{SDIRKbound2}
    \|z_{n+1} - y_{n+1}\| &\leq &
    \|z_n - y_n\|  + 
\varepsilon \dt^2 L  \left(3 \gamma - 1\right).
\end{eqnarray}
\end{subequations}

It would be natural to move on to a method with more stages, for example
the three stage fourth order SDIRK method \cite{CrouzeixBstab}:
 
\begin{eqnarray}\label{SDIRK4} 
z^{(1)} & = & z_n + \frac{1+\alpha}{2} \dt f(z^{(1)})   \nonumber \\ 
z^{(2)} & = & z_n - \frac{\alpha}{2} \dt  f(z^{(1)}) 
+ \frac{1+\alpha}{2} \dt f(z^{(2)}) \nonumber\\
z^{(3)} & = & z_n + (1+\alpha) \dt  f(z^{(1)}) 
-  (1+2 \alpha)  \dt f(z^{(2)}) +\frac{1+\alpha}{2} \dt f(z^{(3)})  \nonumber\\
    z_{n+1} & = & z_n + \frac{\dt}{6 \alpha^2} \big(f(z^{(1)}) + (6 \alpha^2 -2) f(z^{(2)}) + f(z^{(3)})\big) ,
\end{eqnarray}
and its mixed accuracy analog
\begin{eqnarray} \label{MPSDIRK4}
y^{(1)} & = & y_n + \frac{1+\alpha}{2} \dt f_\varepsilon(y^{(1)})  \nonumber \\ 
y^{(2)} & = & y_n - \frac{\alpha}{2} \dt  f(y^{(1)}) 
+ \frac{1+\alpha}{2} \dt f_\varepsilon(y^{(2)}) \nonumber\\
y^{(3)} & = & y_n + (1+\alpha) \dt  f(y^{(1)}) 
-  (1+2 \alpha)  \dt f(y^{(2)}) 
+\frac{1+\alpha}{2} \dt f_\varepsilon(y^{(3)})  \nonumber \\
    y_{n+1} & = & y_n + \frac{\dt}{6 \alpha^2} \big(f(y^{(1)}) 
    + (6 \alpha^2 -2) f(y^{(2)}) + f(y^{(3)})\big) ,
\end{eqnarray}
where $\alpha = \frac{2}{\sqrt{3}} \cos(\frac{\pi}{18}) $.
However, at this point we will move on to a general formulation that includes this method
as well as many others in the class of \eqref{MPDIRK}.

\subsection{General DIRK method}
We now turn to the general case of the 
errors from an $s$-stage method \eqref{errors}.
% Recall that Theorem ? gave us the bound
% \eqref{boundMPcont}:
% \begin{eqnarray*}
% \|z_{n+1} - y_{n+1}\|^2 &\leq &
% \|z_n - y_n\|^2   + 2 \dt^2  L 
% \sum_{i=1}^s b_{i} a_{ii} \varepsilon_i
% \underbrace{\left\| z^{(i)} -  y^{(i)} \right\|}_{\mbox{stage  errors}}  .
%     \end{eqnarray*}
The following lemma
bounds the growth of these stage errors.

\begin{lem} \label{lem2}
Let  $z^{(i)}$ be the $i$th stage of 
\eqref{DIRK} and $y^{(i)}$ be the $i$th stage of  \eqref{MPDIRK} , 
If the perturbation error vector is bounded 
\[ \max_i \| h(y^{(i)}) \| = \vh_i \leq \varepsilon_i \leq \varepsilon \] 
then the  stage error will be bounded by
\begin{eqnarray} \label{StageErrorBound}
\left \| z^{(i)} - y^{(i)}\right\| 
& \leq & K_i \left\|z_n - y_n\right\| +
 \dt C_i  .
\end{eqnarray}
where 
\[  K_i  =  1 + 2 \left( \sum_{\ell =1}^{s-1} 
\left| (\mA- \mAh) \mA^{-1} \right|^\ell
\ve \right)_i \]
and
\[ C_i  =   a_{ii} \vh_i 
+  2 \left( \sum_{\ell =1}^{s-1} \left| (\mA- \mAh) \mA^{-1} \right|^\ell \mAh \vh \right)_i .\]
(The notation $|\cdot|$ here denotes the componentwise absolute value of
the matrix).
\end{lem}

\begin{proof}
From the definitions of \eqref{DIRK}
and \eqref{MPDIRK}, the  vector of 
internal errors is
\begin{eqnarray*} 
 \vz - \vy & = & \left(z_n - y_n\right) \ve  + \dt  \mA \left( f(\vz)  - f(\vy) \right) 
    + \dt \mAh h(\vy) ,
% z_{n+1} -   y_{n+1} & = & z_n - y_n + 
%\vb \left( f(\vz)  - f(\y) \right) ,
\end{eqnarray*}
where $\ve$ is a column vector of ones, and $\mAh$ is a matrix with only the diagonal entries of $\mA$. 
Up to now we have been considering the norm
of values that are scalars, as in $\|z_n - y_n\|$;
we now extend this notation trivially to the vector form, 
where by $\| \vz - \vy \|$  we do not mean
the vector norm, but rather a  vector of vector norm with elements 
$\left\| z^{(i)} - y^{(i)}\right\|$.
We use Proposition \ref{prop1} to give
\begin{eqnarray*}
\|  \vz - \vy  \| & \leq  &
\|  \vz - \vy - 
\dt \mAh (f(\vz)-f(\vy) )  \| \\
& =  & 
\left\| (z_n - y_n) \ve   
+ \dt (\mA - \mAh) (f(\vz)-f(\vy) )
+ \dt  \mAh h(\vy)  \right\| \\
& =  &
\left\|  (z_n - y_n) \ve +  (\mA- \mAh) \mA^{-1} \left(\vz - \vy -   (z_n - y_n) \ve  - \dt \mAh h(\vy) \right)  
+ \dt  \mAh h(\vy)  \right\|
\end{eqnarray*}
where we replaced
\begin{eqnarray*}
\dt  \left( f(\vz) - f(\vy) \right) &=& 
\mA^{-1} \left( \vz - \vy -   
(z_n - y_n) \ve - \dt \mAh h(\vy) \right) .
\end{eqnarray*} 
(Note that if  $a_{11} = 0$, we simply treat the first stage  
as explicit and proceed with the next stages.)
We proceed to bound the error at each stage
\begin{eqnarray*}
\|  \vz - \vy   \| & \leq  & 
\left\|  \left( \mI -(\mA- \mAh) \mA^{-1} \right)  (z_n - y_n) \ve    
\right\| 
+ \left\|   (\mA- \mAh) \mA^{-1}  
(\vz - \vy)    \right\| \\
& + & \dt \left\|   \left( 
\mI - (\mA - \mAh) \mA^{-1}   \right) \mAh h(\vy)    \right\|.
\end{eqnarray*}
We define the matrix 
$\mQ = \left| (\mA- \mAh) \mA^{-1} \right|$ where the 
$| \cdot |$ is taken element-wise. 
Note that $\mQ$  is a strictly lower triangular matrix.
Then we have 
\begin{eqnarray*}
\|   \vz - \vy   \| 
& \leq  & 
 \left\|(z_n - y_n) \right\|  \left(\mI + \mQ \right) \ve  
+ \mQ \left\| \vz - \vy    \right\|
+ \dt   \left( \mI + \mQ \right) \mAh  \vh
\end{eqnarray*}
so that
\begin{eqnarray*}
\| \vz - \vy \| & \leq  & \left\|   (z_n - y_n)  \right\| 
\left(\mI - \mQ \right)^{-1} \left(\mI + \mQ \right) \ve 
+  \dt \left(\mI - \mQ \right)^{-1} 
\left( \mI + \mQ \right) \mAh \vh ,
\end{eqnarray*}
where $\vh$ is a vector that contains the element-wise upper bound   $\left|h(\vy)_i\right| \leq \vh_i $. We observe that 
$ \left(\mI - \mQ \right)^{-1} \left(\mI + \mQ \right) 
= \mI + 2 \sum_{k=1}^{s-1} \mQ^k $, which is 
is a lower triangular matrix with ones on the diagonal and non-negative entries elsewhere.
Define
\begin{eqnarray*}
K_i  =     1 + 2 \left( \sum_{\ell =1}^{s-1} \mQ^\ell \ve \right)_i 
= 1 + 2 \left( \sum_{\ell =1}^{s-1} 
\left| (\mA- \mAh) \mA^{-1} \right|^\ell \ve \right)_i \\
C_i  =   a_{ii} \vh_i +  2 \left( \sum_{\ell =1}^{s-1} \mQ^\ell \mAh \vh \right)_i  
=   a_{ii} \vh_i +  2 \left( \sum_{\ell =1}^{s-1} \left| (\mA- \mAh) \mA^{-1} \right|^\ell \mAh \vh \right)_i 
\end{eqnarray*}
%\begin{equation*} \end{equation*}
so that the  errors at each stage are bounded 
by \eqref{StageErrorBound}.

\end{proof}

This bound on the internal stage errors enables us to state 
the major result of the paper. The following Theorem bounds the
growth of the errors from step to step, depending {\em only} 
on the coefficients of the
method, the stiffness of the problem, 
and the size of the perturbation.

\smallskip

\begin{thm} \label{thm:FinalConv}
Under the conditions in Lemma \ref{lem1} and 
Lemma \ref{lem2}, 
\begin{eqnarray} \label{ErrorGrowth1}
\|z_{n+1} - y_{n+1}\| & \leq & \|z_n - y_n\|
+ \dt^2 L \Theta +  \dt \sqrt{2 \Omega L \dt} ,
\end{eqnarray}
where
\begin{eqnarray*}
\Theta = \sum_{i=1}^s \varepsilon_i b_{i} a_{ii} 
\left(1 + 2 \left( \sum_{\ell =1}^{s-1} \left| (\mA- \mAh) \mA^{-1} \right|^\ell \ve \right)_i \; \right)
\end{eqnarray*}
and
\begin{eqnarray*}
\Omega = \sum_{i=1}^s \varepsilon_i  b_{i} a_{ii}  
\left(  a_{ii} \vh_i +  2 \left( \sum_{\ell =1}^{s-1} \left| (\mA- \mAh) \mA^{-1} \right|^\ell \mAh \vh \right)_i \; \right).
\end{eqnarray*}
If $\dt \geq 2 \Omega/(L \Theta^2)$, this
can be improved:
\begin{eqnarray} \label{ErrorGrowth2}
\|z_{n+1} - y_{n+1}\| & \leq & \|z_n - y_n\|
+   \dt^2 L \Theta . 
\end{eqnarray}
\end{thm}

\begin{proof}
We put the bounds of the internal stages 
\eqref{StageErrorBound} into Equation \eqref{boundMPcont} 
of Lemma \ref{lem1}:
\begin{eqnarray*}
\|z_{n+1} - y_{n+1}\|^2 &\leq &
\|z_n - y_n\|^2  + 2 \dt^2 L \sum_{i=1}^s \varepsilon_i b_{i} a_{ii}
\left\| z^{(i)} -  y^{(i)} \right\|   \\
&\leq &   \|z_n - y_n\|^2  
+ 2 \dt^2 L \sum_{i=1}^s \varepsilon_i  b_{i} a_{ii} 
\left( K_i \left\|z_n - y_n\right\| 
+  \dt C_i   \right)  \\
& = &   \|z_n - y_n\|^2  
+ 2   \dt^2  L   \Theta  \left\|z_n - y_n\right\| + 
2  \dt^3  L    \Omega  \\
& = &   \|z_n - y_n\|^2  
+ 2  \dt^2 L \Theta
\left\|z_n - y_n\right\| 
+  \dt^4 L^2  \Theta^2   +  \dt^3 L  \left( 2 \Omega - \dt L \Theta^2  \right) \\
& = & \left(\|z_n - y_n\| + \dt^2 L \Theta \right)^2 + 
\dt^3 L  \left( 2 \Omega - \dt L \Theta^2  \right) 
\end{eqnarray*}   
where $\Theta = \sum_{i=1}^s \varepsilon_i b_{i} a_{ii} K_i $
and $\Omega = \sum_{i=1}^s \varepsilon_i b_{i} a_{ii}  C_i$.
The bound \eqref{ErrorGrowth1} follows from 
neglecting the final $\dt L \Theta^2  $ term.
If  $\dt$ is large enough we have
$ 2 \Omega - \dt \Theta^2 L  \leq 0 $,
so that can neglect the entire final term
and obtain the bound \eqref{ErrorGrowth2}.
\end{proof}

 \begin{rmk}
    This theorem tells us that we are able to control the final time error
    by ensuring that the perturbation is small compared to the stiffness of the problem,
    the final time, and the time-step. The perturbation vector $\vh$ will depend on many 
    factors, including the size of the problem, the type of perturbation, and the 
    derivatives of $f$. An understanding of the perturbation itself is key to
    determining whether the final time error will be acceptable.
 \end{rmk}
 
\section{Understanding the perturbation errors:  a numerical study using
Burgers' equation} \label{sec:sec3num}
Our primary motivation in this work is to understand the impact of the pollution 
from the mixed accuracy or mixed precision computation of the 
nonlinear implicit stages on the final time solution.  
There are many possible sources of perturbation. For example,
iterative solutions of nonlinear systems are typically performed
as repeated linearizations, and a mixed precision implementation replaces
the repeated solution of a linear system with a low precision version of this  step. 
In this section we numerically investigate  the two sources  of error: 
(1) the errors  resulting from  linearizing the nonlinear 
$f$ using Taylor series and further perturbing this by truncating the inverse operator, 
and (2) the errors resulting from a  mixed precision implementation of a iterative
nonlinear solver. 

Consider the inviscid Burgers' equation
 \begin{eqnarray} \label{eq:Burgers}
     u_t + \left( \frac{1}{2} u^2 \right)_x =0,
 \end{eqnarray} 
 on the domain $x = (0,2\pi)$. The initial conditions and final time will vary 
 depending on our focus. In Section \ref{sec:BurgersLin} we first focus on errors coming from linearization,  where we introduce further perturbation  to show the impact of inaccurate implicit solves. 
Next in Section \ref{sec:BurgersMP} we implement the nonlinear solver with a 
mixed precision approach and  assess the errors resulting from it.

 \subsection{Linearization \& perturbation} \label{sec:BurgersLin}
In this section we consider Burgers' equation \eqref{eq:Burgers}
 initial condition $u(x,0)= \frac{1}{2} + \frac{1}{4} \sin(x)$ 
 and periodic boundary conditions. 
 We semi-discretize this equation in space using a Fourier spectral 
 method differentiation matrix $D_x$. 
 Hence, we aim to solve the differential equation
 \[ \frac{dy}{dt} = f(y) = - \frac{1}{2} D_x y^{2}.\]
 We evolve the solution to final time $T_f = 3.5$ 
using three time-stepping methods:
the mixed-model SDIRK2 \eqref{MPIMR},
 SDIRK3 \eqref{MPSDIRK3}, 
and the SDIRK4  \eqref{MPSDIRK4}.

 The low-accuracy function $f_\epsilon$ is given by a Taylor series linearization
 around $\ybar$
 \begin{eqnarray}\label{linB3}
f_{\epsilon}(y) &=&  f(\bar{y}) + 
    f'(\bar{y}) \left( y - \bar{y} \right) 
    =  - \frac{1}{2} D_x \bar{y}^2
    - D_x \bar{Y} \left( y - \bar{y} \right) ,
    \end{eqnarray}
where we use $\bar{Y}$.

\begin{figure}[H]
    \centering
    \begin{subfigure}[b]{0.32\textwidth}
        \centering        \includegraphics[width=\textwidth]{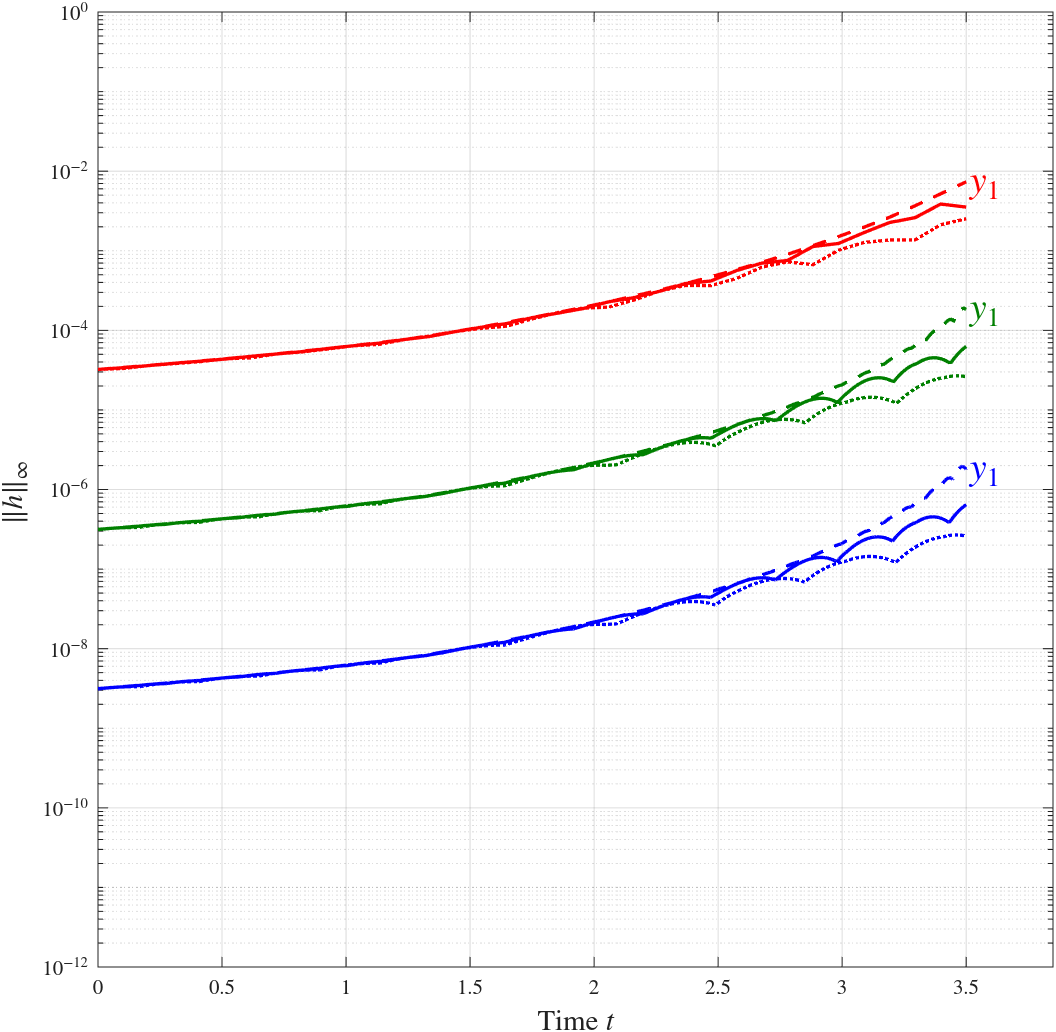}
    \end{subfigure}
    \hfill
   \begin{subfigure}[b]{0.32\textwidth}
        \centering        \includegraphics[width=\textwidth]{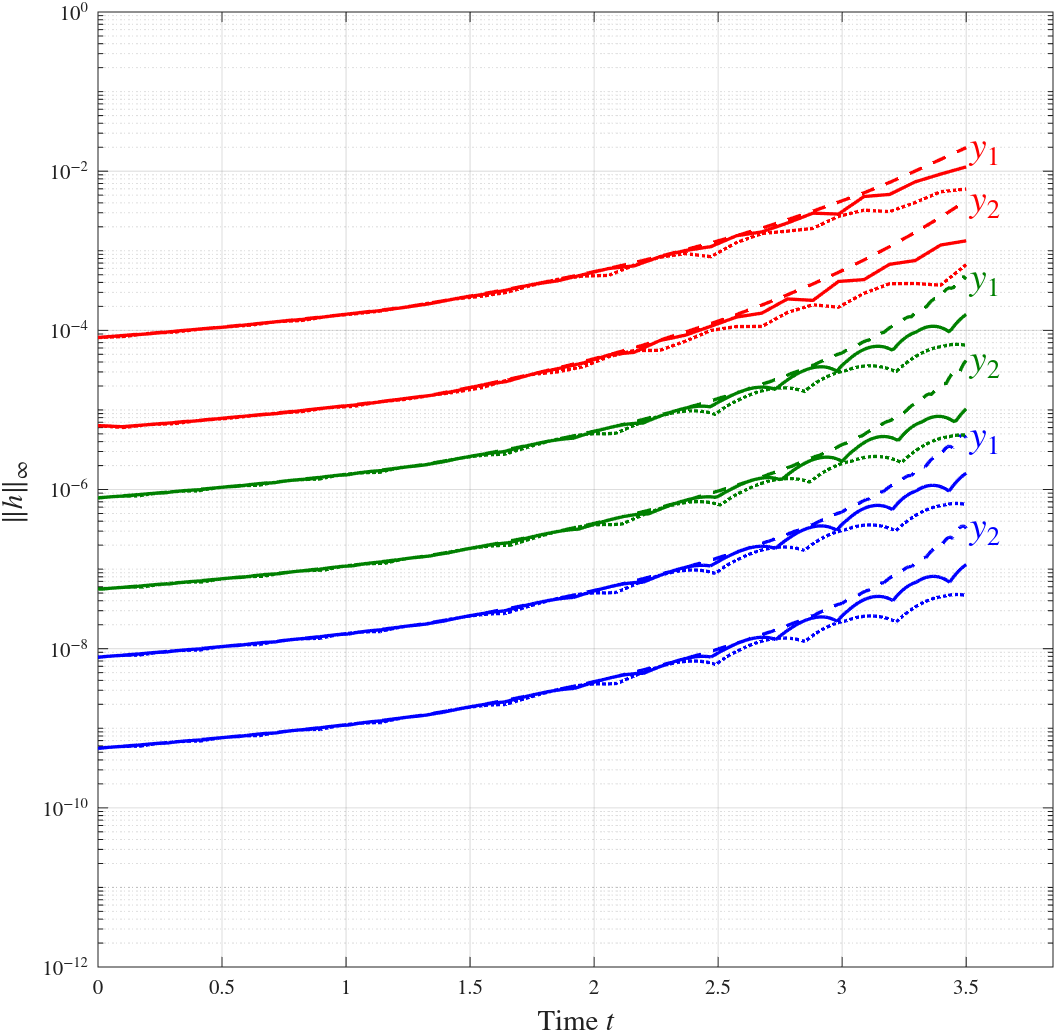}
    \end{subfigure} 
    \hfill
   \begin{subfigure}[b]{0.32\textwidth}
        \centering        \includegraphics[width=\textwidth]{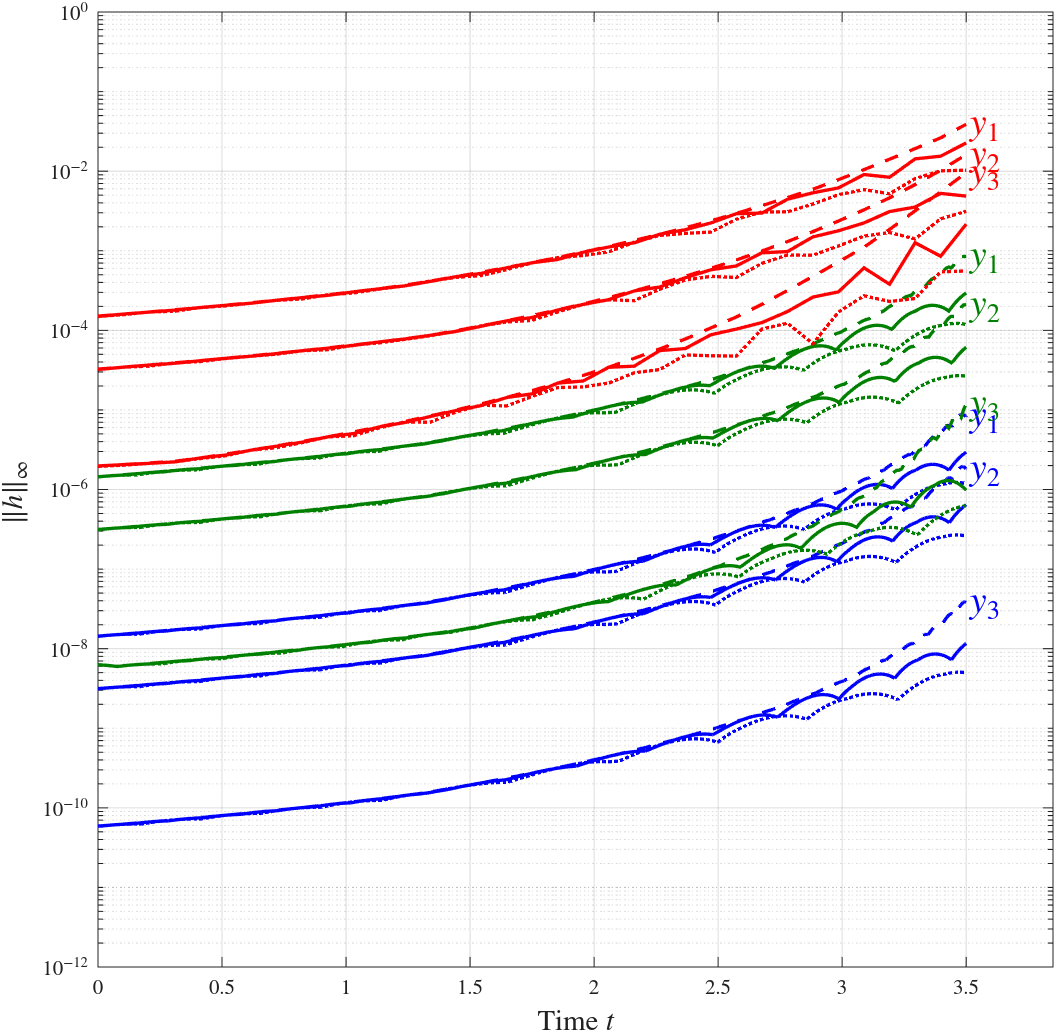}
    \end{subfigure} 
    \caption{The time evolution of $\|h(y^{i})\|_\infty$ for the
    linearized  Burgers' equation with initial condition
    $\frac{1}{2} + \frac{1}{4} \sin(x)$ evolved to time $T_f=3.5$.
    Left: mixed accuracy SDIRK2 \eqref{MPIMR}; 
    Middle: mixed accuracy  SDIRK3 \eqref{MPSDIRK3};
    Right: mixed accuracy  SDIRK4 \eqref{MPSDIRK4}.
    Red lines are $\dt = 0.1$, Green lines $\dt = 0.01$, Blue lines $\dt = 0.001$.
    The dotted lines are for $N_x = 30$, 
  solid lines are $N_x=50$, 
  dashed  lines  $N_x=250$.
    \label{fig:BurgersLinH} }
\end{figure}

Figure \eqref{fig:BurgersLinH} shows the time evolution of $\|h(y^{i})\|_\infty$ 
where
\[ h(y) = - \frac{1}{2} D_x Y y + \frac{1}{2} D_x \bar{Y} \bar{y}
+ D_x \bar{Y} \left( y - \bar{y} \right) =  -  \frac{1}{2} D_x  \left(\bar{Y} -Y
\right)^2 \ve = O(\dt^2). \]
(we use $\bar{y} = u^n$)
for the the mixed accuracy SDIRK2 \eqref{MPIMR}  (left);
SDIRK3 \eqref{MPSDIRK3} (middle);
SDIRK4 \eqref{MPSDIRK4} (right).
Red lines are $\dt = 0.1$, Green lines $\dt = 0.01$, Blue lines $\dt = 0.001$.
The dotted lines are for $N_x = 30$,  solid lines are $N_x=50$, 
dashed  lines  $N_x=250$.
We observe that the biggest impact comes from the value of $\dt$,
and that the size of the perturbation decays, as expected, by a factor of $\dt^2$.
The size of the system makes a difference as well, but it is not a significant difference. 
After a longer time-evolution we see a slight rise  in $\|h\|_\infty$
from $N_x = 30$ to $N_x=50$, and a slightly larger rise to $N_x=250$.
It is also interesting to note that as the solution is more accurate (i.e. the 
order of the time-stepping method is higher) the final stage $\|h(y^{(s)})\|_\infty$
is significantly smaller.

% {\bf Do we do this? I think not. We can delete}
% Next, we look at the evolution of the stage errors
% $z^n - y^n$ for  different values of $N_x$ and $\dt$.
% In Figure \ref{fig:BurgersLinStage} we plot the 
% largest spatial error against over time. Once again, 
% in Figure \ref{fig:BurgersLinStage} (left) we 
% we use $N_x = 50$, and  we plot the value of 
% $log_{10}(\left\|z^n-y^n\right|_\infty)$.
% over time. 
% In blue we have the values of the difference between the solutions from \eqref{IMR} and \eqref{MPIMR}, 
% in red the difference between the solutions from 
% \eqref{SDIRK3} and \eqref{MPSDIRK3}, and in 
% green the difference between the solutions from 
%  \eqref{MPSDIRK4} and its full-accuracy analog (where $f_\varepsilon = f$).
% In Figure \ref{fig:BurgersLinStage} (right) we 
% repeat this for $N_x = 250$.
% We observe {\bf  Put in what we observe.}

\begin{figure}[htb]
    \centering
    \begin{subfigure}[b]{0.32\textwidth}
        \centering        \includegraphics[width=\textwidth]{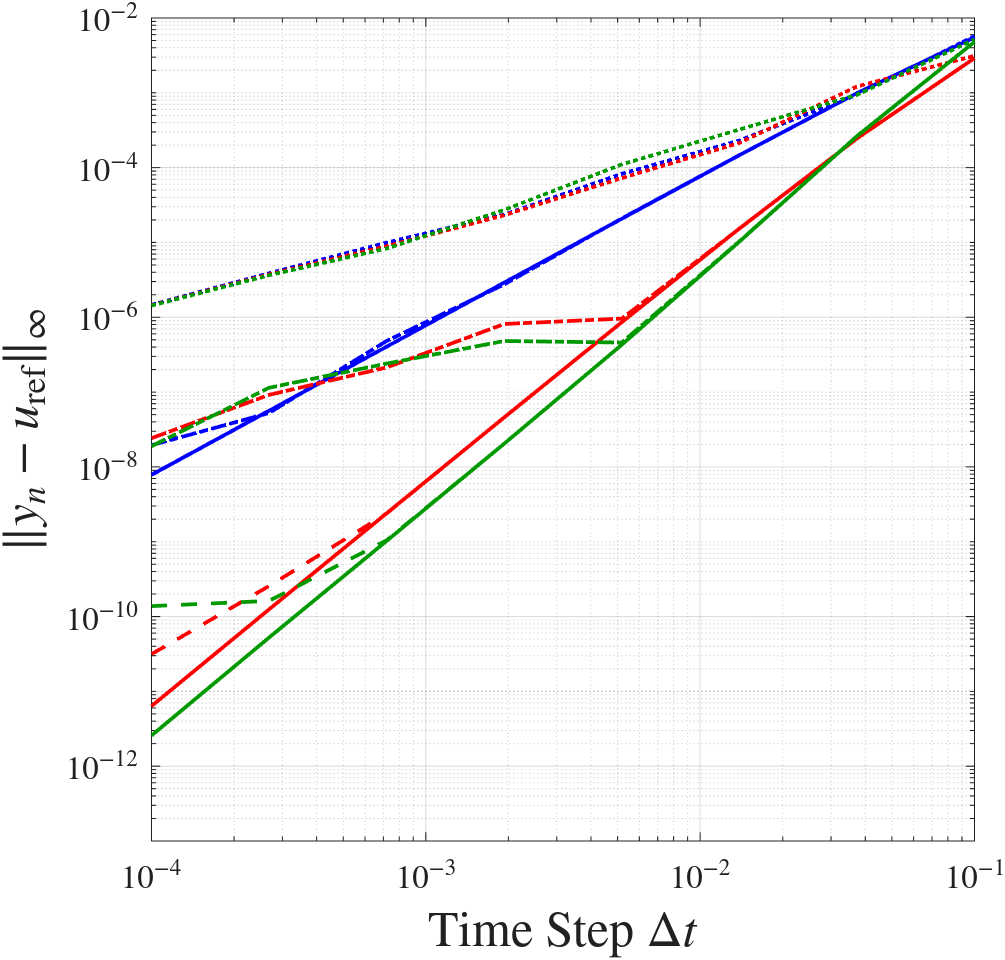}
    \end{subfigure}
    \hfill
   \begin{subfigure}[b]{0.32\textwidth}
        \centering        \includegraphics[width=\textwidth]{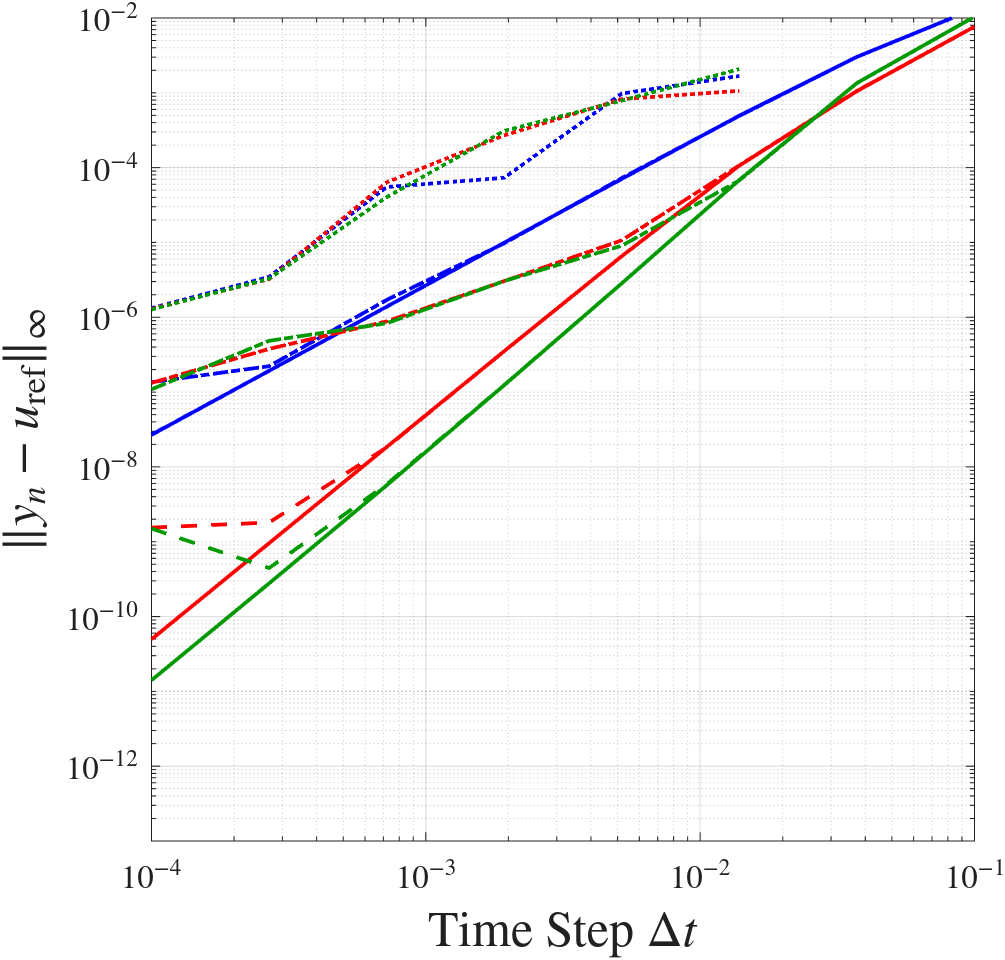}
    \end{subfigure} 
    \hfill
   \begin{subfigure}[b]{0.32\textwidth}
        \centering        \includegraphics[width=\textwidth]{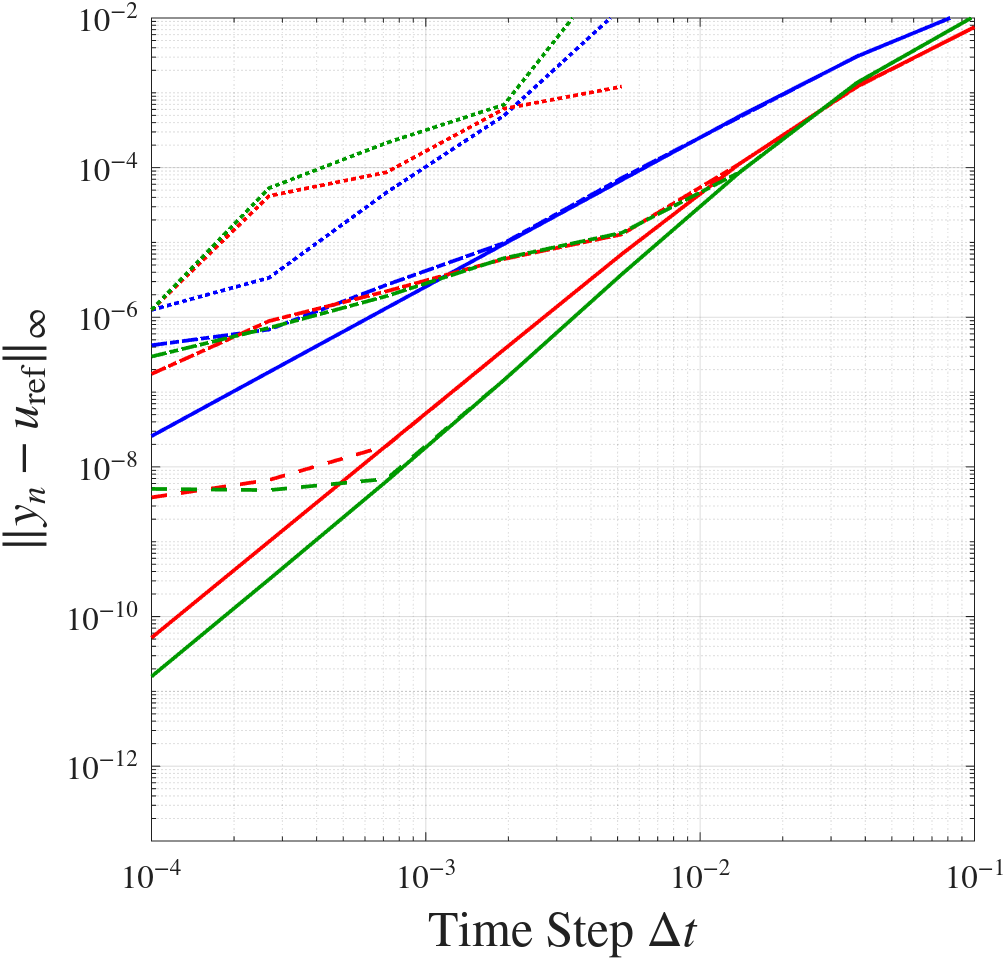}
    \end{subfigure} 
    \caption{The final time maximum norm errors of the linearized and perturbed Burgers' equation,
    compared to a reference solution. 
    In blue we have the mixed accuracy SDIRK2
    \eqref{MPIMR}, in red the mixed accuracy SDIRK3 \eqref{MPSDIRK3},
    and in green the mixed accuracy SDIRK4 \eqref{MPSDIRK4}.
    Left: $N_x = 50$;  Middle $N_x = 250$; Right $N_x = 450$.
    Solid lines are only linearized, but not perturbed;
    Dashed lines have a perturbation of $\epsilon = 10^{-6}$;
    Dash-dot lines have a perturbation of $\epsilon = 10^{-4}$;
    Dotted lines have a perturbation of $\epsilon = 10^{-2}$.
    \label{fig:BurgersLinFT} }
\end{figure}

We now turn to look at the impact of the perturbation 
on the final time errors.  In Figure \ref{fig:BurgersLinFT}
we plot $log_{10}$ of the final time errors of the methods SDIRK2
\eqref{MPIMR} (blue), SDIRK3 \eqref{MPSDIRK3} (red), and SDIRK4
\eqref{MPSDIRK4} (green) when compared to  a reference solution, 
plotted against  $\dt$ (x-axis),
for values of $N_x =50, 250, 450$ (left, middle, right). 
We see that in the absence of a perturbation, 
the linearization error is not apparent 
for the second order SDIRK2 and the third order SDIRK3 methods. This is because 
the linearization has an error of $\varepsilon = O(\dt^2)$,
and, as expected by Theorem \ref{thm:FinalConv} the final time error is expected to be 
$O(\varepsilon \dt L) = O(\dt^3) $, which is less or equal to the 
order of these two methods. However, for the fourth order SDIRK4 method
the linearization error is dominant, so we only see third order convergence.

Next, we want to understand the impact of perturbation errors in addition to the 
linearization error. We perturb the inverse matrix  $(I - a_{ii} \dt D_x \bar{Y})^{-1}$
by chopping it off after a set number of digits $d$, 
leading to  a perturbation of $\epsilon = 10^{-d}$.
In Figure \ref{fig:BurgersLinFT} we show the impact of these perturbations,
at the level of 
$\epsilon = 10^{-2}$ (dotted),
$\epsilon = 10^{-4}$ (dash-dot),
and $\epsilon = 10^{-6}$ (dashed).
We observe that for the smallest perturbation $\epsilon = 10^{-6}$
the impact is not seen for larger $\dt$, but as $\dt$ gets smaller
the convergence rate drops to first order, and eventually 
saturates at the level of $\epsilon \; \dt$.
This happens sooner and is more evident as we have more points in space, i.e. as the 
problem is stiffer and the impact of the polluted matrix multiplication increases.
As the perturbation gets larger $\epsilon = 10^{-4}$ (dash-dot)
we see clear first order convergence that starts earlier as $N_x$ is larger.
When we use a large perturbation $\epsilon = 10^{-2}$ (dotted), we have large errors that are first order for all $\dt$ when $N_x$ is smaller, and grow less stable 
(the lines abruptly end) as $N_x$ gets larger.

We see that Theorem \ref{thm:FinalConv} explains the growth of
these errors seen in this linearization and perturbation example.
In the next section, we look at a true mixed precision implementation of 
a similar Burgers' equation.

\subsection{Mixed precision implementation with an iterative solver}
\label{sec:BurgersMP}

In this section, we consider a mixed precision 
implementation of the nonlinear solver. For this case, we
use the Burgers' equation  \eqref{eq:Burgers}
 with   initial condition $u(x,0)= \sin(x)$  and periodic boundary conditions.
 We evolve the solution to final time $T_f =0.7$.

The overall motivation for this study is the use of mixed precision arithmetic to accelerate the computation. The most expensive part of the computation involves the iterative solution of the linearized system, as in Newton's method. The most expensive part of this iteration is the repeated solution of a linear operator. For Newton's method this linear operator is 
obtained from repeated Taylor series linearizations, as those performed in Section \ref{sec:BurgersLin}. In this section, we combine repeated Taylor series  linearization with a mixed precision computation of the inverse linear problem to show the impact of the combined perturbation.

Each implicit stage has the  general form:
$ y = y_{exp} + \alpha \dt f(y).$
We solve this iteratively, by making two replacements at each iteration: 
First, we replace $f(y)$ with 
$f_{lin}(y) = f(\bar{y}) +  f'(\ybar) \left( y - \bar{y} \right) $,
with the appropriate  $\ybar$  at each iterate.
Next, we solve the resulting system in mixed precision. 

\begin{framed}
\noindent{\bf Mixed precision algorithm:}
Select an initial value $y_{[0]}$, typically  
    $y_{[0]} = y_{exp}$. Now, for each iterate $k$ starting from $k=0$:
\begin{enumerate}
\item Replace $f(y_{[k]})$ with  
   $ f_{lin}(y) = f(y_{[k]}) +  f'(y_{[k]}) \left( y - y_{[k]} \right) .$
\item Plug in:
    $ y = y_{exp} + \alpha \dt \left(  f(y_{[k]}) +  
    f'(y_{[k]}) \left( y - f'(y_{[k]}) \right) \right). $
\item Compute 
    $y_{e} = y_{exp} + \alpha \dt f(\bar{y}) - \alpha \dt f'(y_{[k]}) y_{[k]}$,
    in high precision,  and cast it down to low precision $y_{e}^\epsilon$.
\item   Compute $\mathcal{J} =  \mI - \alpha \dt f'(y_{[k]}) $
    we cast it down to low precision  $\mathcal{J}^\epsilon$.
\item   Solve in low precision 
    $ \mathcal{J}^\epsilon \tilde{y}^\epsilon = y_{e}^\epsilon .$
\item  Cast  $\tilde{y}^\epsilon$ up to high precision $\tilde{y}$.
\item Plug this back in to the high precision operator to obtain the high precision iterate:
    $y_{[k+1]} = y_{e} + \alpha \dt f'(y_{[k]})\tilde{y} .$
\end{enumerate}
{\em Note that if the entire stage is performed in low precision rather than 
just the implicit solve then the error we obtain will depend on 
$\epsilon_{prec}$ rather than $\dt \epsilon_{prec}$.
The resulting error in $y_{[k+1]}$ is a combination of  $\dt \epsilon_{prec}$ and  $\dt \varepsilon_{lin}$ where $\varepsilon_{lin}$ is the error from linearization 
and $\epsilon_{prec}$ is the low precision error.}
\end{framed}

\begin{figure}[!htb]
    \centering
    % Row 1
    \begin{subfigure}[b]{0.30\textwidth}
        \centering        \includegraphics[width=\textwidth]{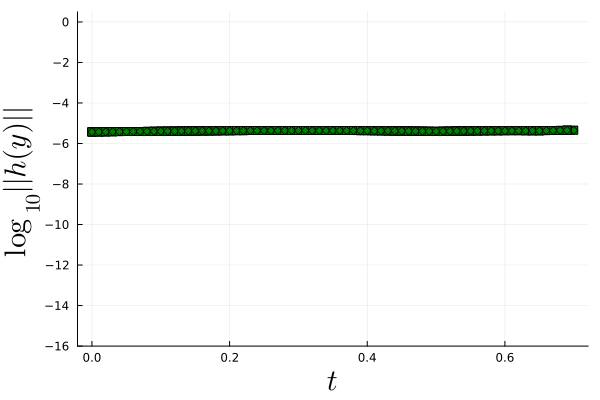}
    \end{subfigure}
    \hfill
    % Row 2
    \begin{subfigure}[b]{0.30\textwidth}
        \centering   \includegraphics[width=\textwidth]{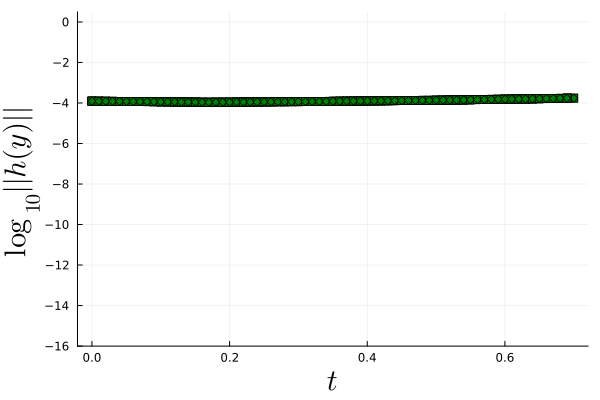}
    \end{subfigure}    \hfill
    % Row 2
    \begin{subfigure}[b]{0.30\textwidth}
        \centering   \includegraphics[width=\textwidth]{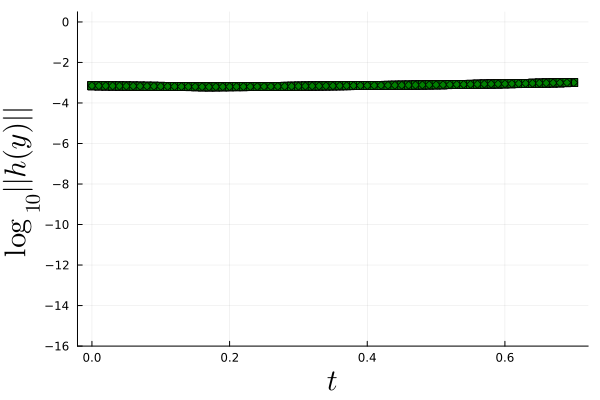}
    \end{subfigure}\\
        \begin{subfigure}[b]{0.30\textwidth}
        \centering        \includegraphics[width=\textwidth]{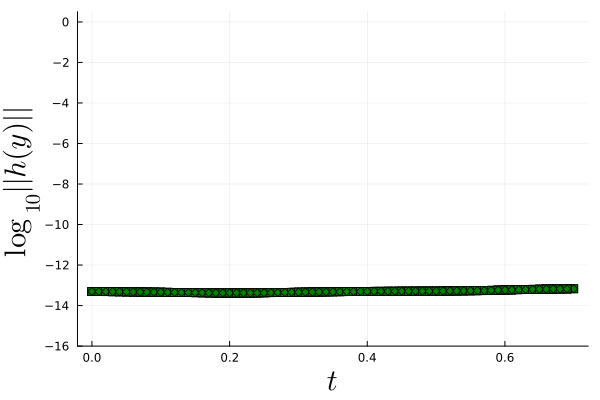}
    \end{subfigure}
    \hfill
    % Row 2
    \begin{subfigure}[b]{0.30\textwidth}
        \centering   \includegraphics[width=\textwidth]{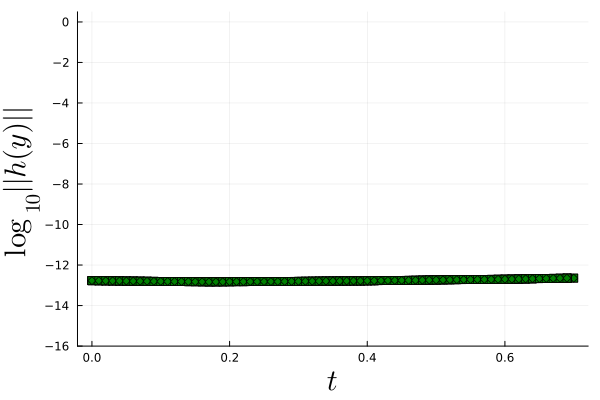}
    \end{subfigure}    \hfill
    % Row 2
\begin{subfigure}[b]{0.30\textwidth}  \centering 
 \includegraphics[width=\textwidth]{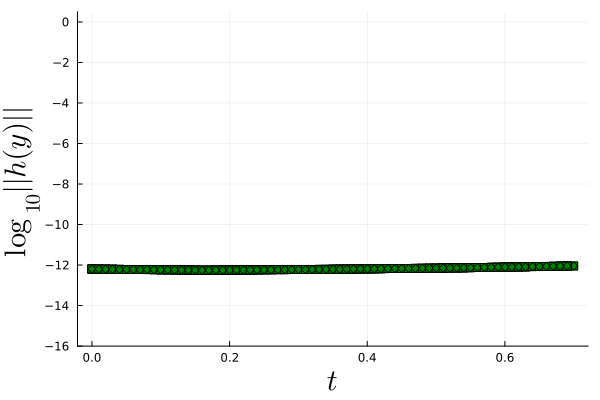}
    \end{subfigure}
    \caption{ Time evolution of $\left| h \right|_\infty$ for mixed precision 
    Burgers' equation for double/single (64/32) on top
    and quad/double (128/64) on bottom. 
    Blue: SDIRK2; red: SDIRK3; green: SDIRK4.
    Dotted: $\Delta t = 10^{-2}$; dash-dotted: $\Delta t = 10^{-3}$; dash: $\Delta t = 10^{-4}$. Left: $N_x = 50$; center: $N_x = 100$; right: $Nx = 200$. Each marker corresponds to a different stage.}\label{figmp1}
\end{figure}

Figure \ref{figmp1} shows the time evolution of 
$\left\| h(y^{(i)})\right\|_\infty$ for 
the mixed double/single (top) and quad/double (bottom)
with  $N_x = 50$ (left),  $N_x = 100$ (center), and  $N_x = 200$ (right).
The lines for the three methods SDIRK2 \eqref{MPIMR}, 
SDIRK3 \eqref{MPSDIRK3}, and SDIRK4 \eqref{MPSDIRK4} overlap. 
The perturbation most strongly depends
on the precision level, with the  double/single values near $10^{-4}$
and the quad/double values between  $10^{-14}$ and $10^{-12}$.
There is also a slight dependence on the number of points:
for double/single the value is slightly below $10^{-5}$
for the $N_x = 50$ case, which rises to above $10^{-4}$ for $N_x=200$.
 For quad/double we see the rise from near 
$10^{-14}$ to $10^{-12}$ as $N_x$ growth.
When we have a larger system, more lower precision
terms are  being multiplied in the matrix-vector operations,
causing  roundoff errors to accumulate.   
\begin{figure}[H]
    \centering
    % --- IMR Plots ---
    \begin{subfigure}[b]{0.48\textwidth}
        \centering
        \includegraphics[width=\textwidth]{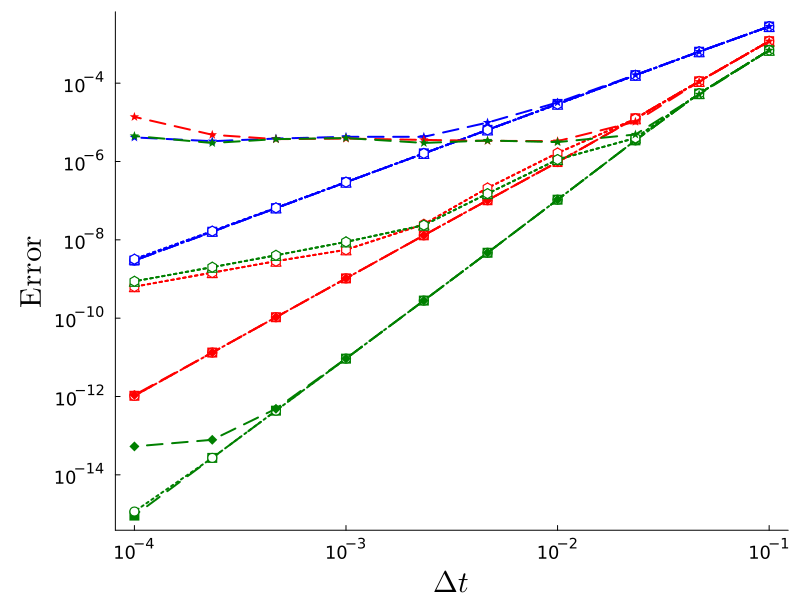}
    \end{subfigure}
    \hfill
    \begin{subfigure}[b]{0.48\textwidth}
        \centering
        \includegraphics[width=\textwidth]{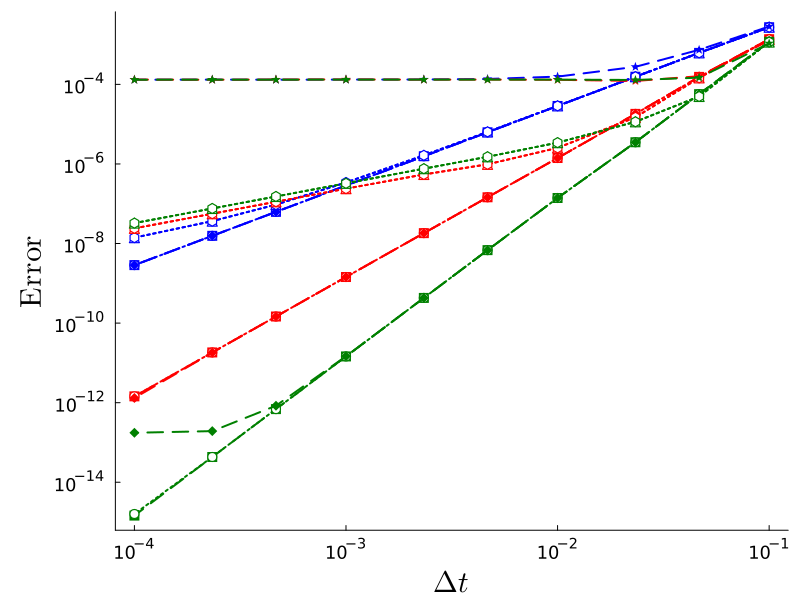}
    \end{subfigure}\\
        % --- IMR Plots ---
    \begin{subfigure}[b]{0.48\textwidth}
        \centering
        \includegraphics[width=\textwidth]{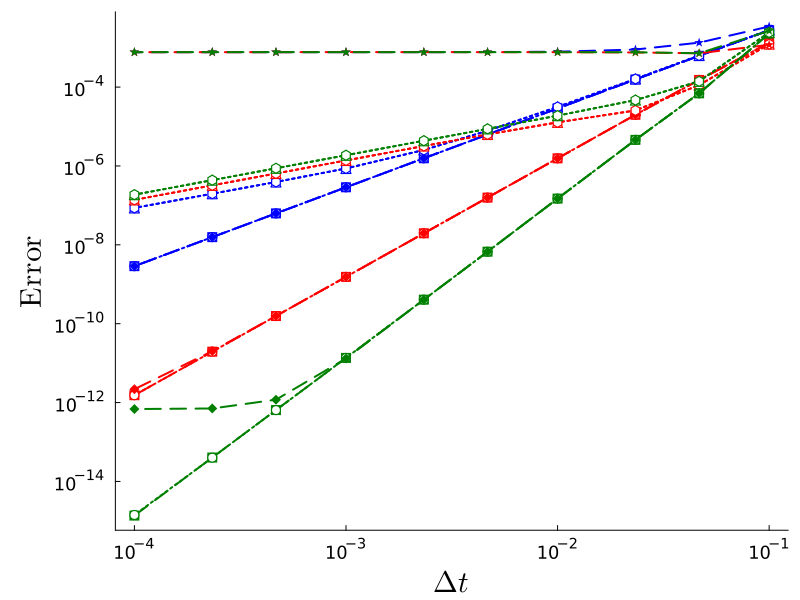}
    \end{subfigure}
    \hfill
    \begin{subfigure}[b]{0.48\textwidth}
        \centering
        % The image is now perfectly centered in the right half of the page.
        % You can change 0.8 to whatever size looks best to you.
        \includegraphics[width=0.52\textwidth]{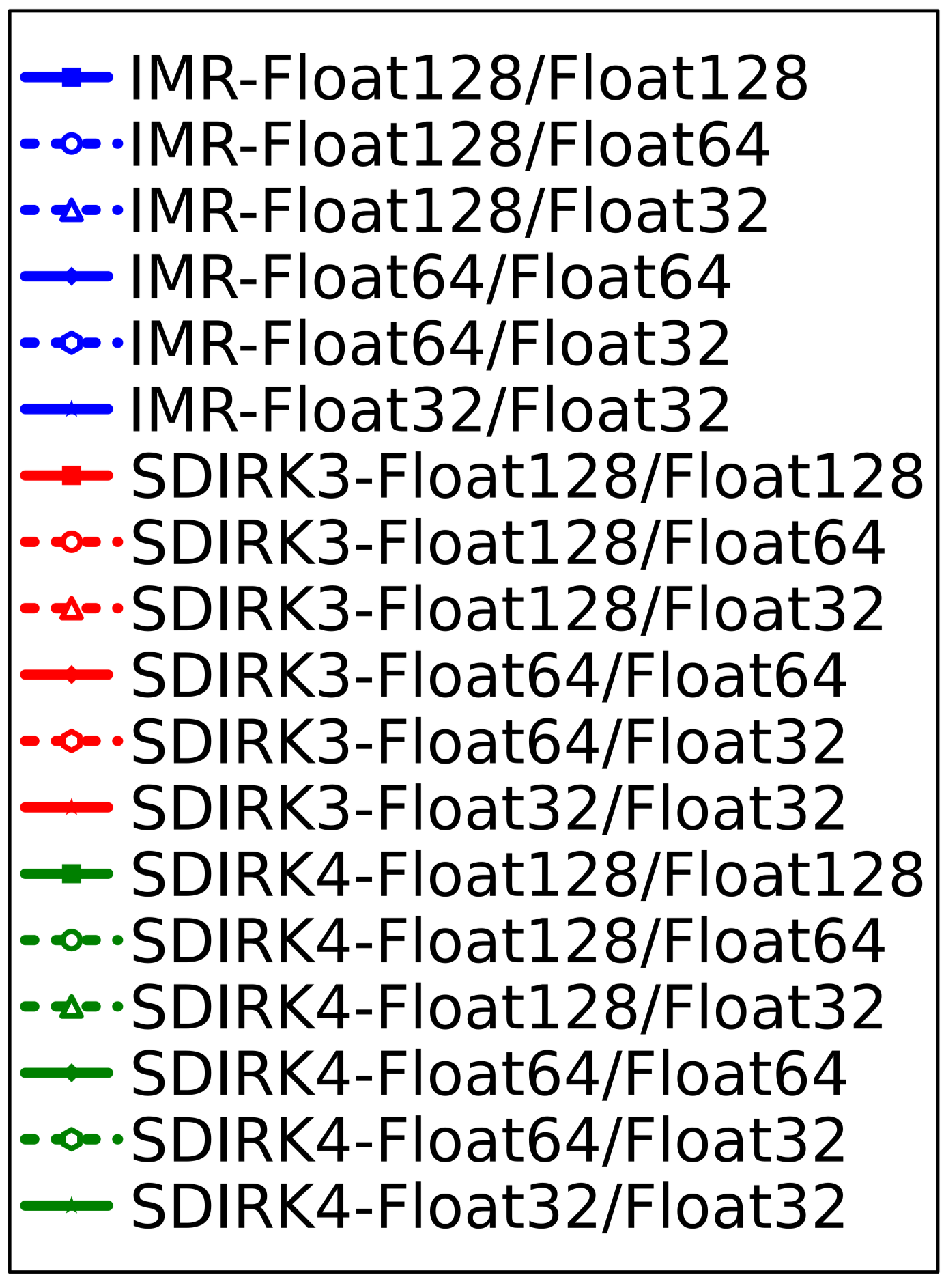}
    \end{subfigure}
    \caption{Mixed Precision Burgers' final time errors from  SDIRK2
\eqref{MPIMR} (blue), SDIRK3 \eqref{MPSDIRK3} (red), and
SDIRK4 \eqref{MPSDIRK4} (green). The errors are computed compared to 
a reference solution, and plotted at different values of $\dt$.
We use quad precision (128), double precision (64),
and single precision (32). 
    Top Left: $Nx = 50$. Top Right: $N_x = 100$. Bottom Left: $N_x = 200$.
    \label{fig:BurgersMPFT}}
    \end{figure}
In Figure \ref{fig:BurgersMPFT} we show the 
$log_{10}$ of the final time errors of the methods SDIRK2
\eqref{MPIMR} (blue), SDIRK3 \eqref{MPSDIRK3} (red), and
SDIRK4 \eqref{MPSDIRK4} (green) when compared to 
a reference solution, at different values of $\dt$.
We use quad precision (128), double precision (64),
and single precision (32). On the top left we show this for 
$N_x=50$ and on the top right for $N_x=100$, the bottom left is 
$N_x=200$.
We observe that the low precision takes over the solution and eventually 
destroys its quality. The mixed precision produces, as expected, first order
errors eventually. We note a  very strong dependence on the size of the system
both for the low precision and the mixed precision errors.
This highlights that the value of $\varepsilon$ is not the same as the machine precision value 
$\epsilon_{prec}$ but is a complex value that in this case is impacted by the precision,
the buildup of errors over each stage, and the size of the system.
This buildup of errors causes the mixed precision higher order methods to have {\em less}
accuracy than the mixed precision lower order methods for sufficiently large $N_x$ (e.g. 
$N_x = 200$).

\section{Stability and accuracy of corrections}
\label{sec:corr}

In the sections above we investigated the accuracy  and stability  of perturbed DIRK methods. 
We showed that the error resulting from replacing $f$ with $f_\epsilon$  
looks like $O(\varepsilon \dt  L T_f)$  at the final time $T_f$. 
This predictable behavior of the error, that does not grow as we increase the number of time-steps, is indication that this approach is stable.  Furthermore, the error grows only
linearly (not exponentially) with final time, which is advantageous. 
However, we note that identifying the  value of $\varepsilon$ 
is not always straightforward as it may depend on the size $N_x$ of the system,
as well as the precision of the implementation, and even the stiffness of the problem.

The first order error that enters from this perturbation will reduce the accuracy of the problem.
Furthermore, the error term $\varepsilon \dt  L$  means that a stiffer problem  (larger $L$)  
will  require a smaller time-step $\dt$ or a smaller perturbation $\varepsilon$ 
to maintain  stability.    We would like to  improve the accuracy and stability  
of the  perturbed DIRK method, without significantly adding to the computational cost.
In  this section we explore the use of  stabilized corrections to  improve the accuracy  of the  perturbed methods, without adversely impacting their stability.
 
\subsection{Stabilizing the explicit correction approach}
Explicit corrections were presented in \cite{Grant2022}, 
to improve the accuracy of the  perturbed Runge--Kutta method. The idea is to use explicit highly accurate corrections to mitigate the impact of the perturbation in the implicit step. 
For any  $p$ order implicit method \eqref{MPDIRK} we define the explicit correction method  with $p-1$ correction terms
\begin{subequations} \label{MPDIRKexpC}
\begin{align} 
\mbox{$i=1, . . ., s$:} 
& \left\{\begin{aligned} 
 y^{(i)}_{[0]} & =  y_n + 
    \dt \left( \sum_{j=1}^{i-1} a_{ij} f(y^{(j)}_{[p-1]}))
    + a_{ii} f_{\varepsilon}(y^{(i)}_{[0]}) \right) \\
 y^{(i)}_{[k]} & =  y_n + 
    \dt \left( \sum_{j=1}^{i-1} a_{ij} f(y^{(j)}_{[p-1]})
    + a_{ii} f(y^{(i)}_{[k-1]}) \right) 
    \; \; \; \; k=1, . . ., p-1\\
    \end{aligned} \right. \\
& y_{n+1}  =  y_n + \dt \sum_{i=1}^s b_{i} f(y^{(i)}_{[p-1]}) .
\end{align}
\end{subequations}
Each correction term mitigates the perturbation error 
by $\dt$, as was shown in \cite{Grant2022}
by writing the method with the corrections in an
augmented matrix form.

These inexpensive explicit computations treat $f$ more accurately.   
The implicit solves are expected to be computationally dominant, 
even when performed with lower accuracy, so we expect that  the gain in accuracy 
will be well-worth the extra few cheap explicit stages. This was verified in \cite{Burnett1, Burnett2}, where 
the accuracy and stability of explicit corrections for mixed precision were investigated numerically.
However, these explicit corrections shrink the region of linear stability and may introduce significant 
instability for larger values of  $\dt$.

% In the next  subsections we explore explicit corrections from three perpectives. These allow us to understand the resulting instability and set the stage to improve upon them using an implicit correction approach.

The explicit corrections can be analyzed for both stability and accuracy as a fixed point iteration. For any implicit stage
\[ y = y_{exp} + \alpha \dt f(y)\]
we write the  explicit corrections
\begin{eqnarray*}
  y_{[k+1]} & = & y_{exp}  + \alpha \dt f(y_{[k]}) 
\end{eqnarray*}
Replacing $y_{exp} = y - \alpha \dt f(y)$ we get
\begin{eqnarray*}
\left\| y_{[k+1]} -  y \right\|
& = & \left\| y_{exp}  + \alpha \dt f(y_{[k]}) - y \right\|\\
& = &  \left\|y - \alpha \dt f(y) + \alpha \dt f(y_{[k]}) - y\right\| \\
& = & \alpha \dt \left\| \left( f(y_{[k]}) - f(y) \right)\right\| \\
& \leq & \alpha \dt  L  \left\|y_{[k]}  - y \right\|.
\end{eqnarray*}
This process converges when $\alpha \dt L \leq 1$. 
However, if $\alpha \dt L > 1$ these corrections
may  cause instability.

This understanding of the explicit corrections points 
us to a stabilized correction approach: 
we want to add a term that will balance out the stiffness $L$ while retaining the improvement in accuracy. 
This suggests the following stabilized correction strategy:
\begin{eqnarray} \label{CorrStab1}
y_{[k+1]} & = &
\underbrace{y_{exp} 
+ \alpha \dt f(y_{[k]})}_{\mbox{explicit correction}}
+ \underbrace{\alpha \dt J  \left(y_{[k+1]} - y_{[k]} \right)}_{\mbox{stabilization}} ,
\end{eqnarray}
where the matrix $J$ will be chosen
so that the resulting iteration is stable.
Once again, we can understand this using a  fixed point analysis. We have 
\[ y_{[k+1]}  = \left( I - \alpha \dt J \right)^{-1} 
 \left( y_{exp}  + \alpha \dt f(y_{[k]}) -   \alpha \dt J y_{[k]} \right) = G(y_{[k]}) \]
so that, by the fixed point theorem,  we expect this to converge when $\| G'\| \leq 1$
\[  \left\| \alpha \dt \left( I - \alpha \dt J \right)^{-1}  \left( f'(w) - J  \right) \right\| \leq 1.\]
This is promising for several reasons. 
First, we expect that $f'(w) - J $ will not be too large if we select $J$
    close to the Jacobian. Second, the term $\left( I - \alpha \dt J \right)^{-1}$ should damp out the terms it multiplies. Finally, the entire value is multiplied by $\alpha \dt$ which will shrink it
    further.

To guide the choice of $J$, observe that
\begin{eqnarray*} 
 y_{[k+1]}  &= & \left( I - \alpha \dt J \right)^{-1} 
 \left( y_{exp}  + \alpha \dt f(y_{[k]}) -   \alpha \dt J y_{[k]} \right) \\
 &= &  \left( I - \alpha \dt J\right)^{-1} 
  \left( y_{[k]} -   \alpha \dt J y_{[k]} + y_{exp}  + \alpha \dt f(y_{[k]}) -y_{[k]} \right) \\
&= &   y_{[k]}
 + \left( I - \alpha \dt J \right)^{-1}
 \left(y_{exp}  + \alpha \dt f(y_{[k]}) -  y_{[k]}
 \right)  \\
&= & y_{[k]}
 + \left( I - \alpha \dt J\right)^{-1}
 \left(y - \alpha \dt f(y) + \alpha \dt f(y_{[k]}) -  y_{[k]}
 \right)  \\
 &= & y_{[k]}
 - \left( I - \alpha \dt J \right)^{-1}
 \left( I - \alpha \dt Q_k \right) \left(y_{[k]}-y\right) 
\end{eqnarray*}
where $Q_k (y_{[k]}-y) = f(y_{[k]})-f(y) $ so that
\[  y_{[k+1]}  - y = y_{[k]} - y 
- \left( I - \alpha \dt J \right)^{-1}
 \left( I - \alpha \dt Q_k \right) \left(y_{[k]}-y\right).
\]
Rearranging, we get:
\[  y_{[k+1]}  - y = 
 \alpha \dt  \left( I - \alpha \dt J \right)^{-1}
 \left( Q_k - J \right) \left(y_{[k]}-y\right) .\]
The key is that we want to choose $Q_k-J $ to be small,
and moreover to be made smaller by 
$\left( I - \alpha \dt J \right)^{-1}$.
% Consider that
% \[ f(y_{[k]})-f(y) = f'(\xi_k) (y_{[k]}-y) \]
% for some $\xi_k$ in the interval $(y_k,y)$,
% so that $Q_k = f'(\xi_k)$. 
So we want
\begin{eqnarray} \label{CorCondition}
    \left\| \left( I - \alpha \dt J \right)^{-1} \left( Q_k - J \right) 
\left(y_{[k]}-y\right) \right\| \leq \| y_{[k]}-y \|.
\end{eqnarray} 

Note that this condition is stricter than needed for convergence; it ensures that not only do we converge but we pick up a factor of $\dt$ at each iterate. 
In practice, the method may still converge if this 
condition is violated. However, if we can design $J$ to satisfy this condition, we expect to pick up an $O(\dt)$
at each iteration.

Many approaches may accomplish this. For example,
we can select $J = f'(\eta_k)$ where $\eta_k$ is  
some point in a small interval near $y_k$. 
% Recalling that we can also choose $Q_k = f'(\xi_k)$  
% where $\xi_k$ is some point between $y_k$ and $y$.
% In this case, we have
% \[ Q_k  - J = f'(\xi_k) - f'(\eta_k) = f''(\phi) \left(\xi_k - \eta_k \right) = O(\dt).\]  
% Ideally, we want $I - \alpha \dt J$  to be inexpensive to invert, or that an  inverse can be precomputed and possibly updated in an inexpensive way.
To make this approach efficient  we also require that the  corrections 
do not significantly increase the computational cost.  This can be accomplished,
for example, if $(I - \alpha \dt J)^{-1}$ can be precomputed or if it is  inexpensive to invert at each time-step (e.g. a ,tri-diagonal, or lower triangular matrix).
In  Sections \ref{sec:phi} and  \ref{sec:NumCor} we explore and test different strategies to select 
$\Phi = \left( I - \alpha \dt J \right)^{-1}$ 
that  stabilize the method and allow for rapid and efficient corrections.

\subsection{Analyzing the stabilized corrections as 
a time-stepping method}
We can use the theory in Section \ref{sec:stability} 
to understand the impact of corrections on the accuracy, and on the stability as well.
Consider a DIRK method with the stabilized correction approach:

\begin{align} \label{MPDIRKimpC}
\mbox{$i=1, . . ., s$:} 
& \left\{\begin{aligned} 
 y^{(i)}_{[0]} & =  y_n + 
    \dt \left( \sum_{j=1}^{i-1} a_{ij} f(y^{(j)}_{[p-1]}) \right) 
     + a_{ii} \dt f_{\varepsilon}(y^{(i)}_{[0]}) 
\nonumber \\
 y^{(i)}_{[k]} & =  y_n + 
    \dt \left( \sum_{j=1}^{i-1} a_{ij} f(y^{(j)}_{[p-1]})\right)  \\
    & + a_{ii} \dt f(y^{(i)}_{[k-1]})  
    + a_{ii} \dt J \left(y^{(i)}_{[k]} - y^{(i)}_{[k-1]} \right)
    \; \; \; \; k=1, . . ., p-1 \nonumber\\
    \end{aligned} \right. \nonumber \\
& y_{n+1}  =  y_n + \dt \sum_{i=1}^s b_{i} f(y^{(i)}_{[p-1]}) .
\end{align}
Alternatively, we can express the intermediate stages as 

\begin{eqnarray} \label{MPDIRKimpC2}
y^{(i)}_{[k]} & = & y_n + \dt  \sum_{j=1}^{i-1} a_{ij} 
f(y^{(i)}_{[p-1]}) +
\dt a_{ii} f(y^{(j)}_{[k]})   - \dt a_{ii}  h_{[k]}^{(i)}
%    y_{n+1} & = & y_n + \dt \sum_{i=1}^s b_{i} f(y^{(i)}) .
\end{eqnarray}
where, as before:
\[  h_{[0]}^{(i)}  =  f(y^{(i)}_{[0]}) -f_\varepsilon(y^{(i)}_{[0]})    = O(\varepsilon) \]
and for $k>0$
\[  h_{[k]}^{(i)}  =  f(y^{(i)}_{[k]}) - \left(  f(y^{(i)}_{[k-1]}) + J_k \left(y^{(i)}_{[k]} - y^{(i)}_{[k-1]}\right)  \right) = O(\varepsilon \dt^{k}),\]
if \eqref{CorCondition} is satisfied.

% The iterative process converges if $\dt$ is small enough compared to $f''$:
% \[ y_{k+1} - y  = (y_{k} - y) \times O(\dt^2) =O(\dt^{2(k+1)+1}) . \]
% The idea is that each implicit stage 
% is computed multiple times, and each time the value of $h_{[k]}$ 
% shrinks by an order of $\dt^2$:

To write this type of method in Butcher form, 
we stack the $s$ stages with their $p-1$ corrections 
and represent the  stage coefficients in the $(p\times s) \times (p\times s)$ matrix:
\[ \mA = \left(
\begin{array}{cccccccccc}
a_{11} & \vdots & 0 & 0 & 0 &\vdots & 0 & ... & ...& 0 \\
0 & \ddots &  0 & 0 & 0 &\vdots & 0 &... & ... & 0 \\
0 & \vdots & a_{11} &  0 & 0 &\vdots & 0 & ... & ... & 0 \\
0 & \vdots & a_{21} & a_{22} & 0 &\vdots  & 0  & ... & ... & 0 \\ 
0 & \vdots & a_{21} & 0 & a_{22} &\vdots & ... & ... & ...& ...\\ 
... & ... & ... & ... & ... &\vdots  & ... & ... & ... & ...\\
0 & \vdots & a_{s1} & ... &  ... & \vdots & a_{s,s-1}  & a_{ss} & 0 & 0 \\
0 & \vdots & a_{s1} & ... &  ... & \vdots & a_{s,s-1} & 0 & a_{ss} & 0 \\
0 & \vdots & a_{s1} & ... &  ... & \vdots & a_{s,s-1} & 0 & 0 & a_{ss} \\
\end{array}
\right)
\]
and $\vb$ is a vector of length $p \times s$,  
\[\vb = \left( \underbrace{0, ... 0 ,b_1}, \underbrace{0, ... 0 , b_2}, ...,  \underbrace{0, ... 0 , b_s} \right).\]
Note that the structure is $(p-1)$ zeros followed by a nonzero value,
so that the nonzero values correspond to each 
final corrected stage $ y^{(i)}_{[p-1]}$.

When we apply Theorem \ref{thm:FinalConv} to this case 
we have
\[ \|z_{n+1} - y_{n+1}\| \leq \|z_n - y_n\|
+ \dt^2 L \Theta +  \dt \sqrt{2 \Omega L \dt} ,\]
where, due to all the $\vh_i$ values that are zeroed out,
and the fact that only the fully corrected stage 
contributes to the error, we have:
\[  \Theta = O(\varepsilon \dt^{p-1}) \; \; \; 
\mbox{and}  \; \; \; 
 \Omega =   O(\varepsilon^2 \dt^{2 p-2}) . \]
This allows us to conclude that we will see an overall final time error
of  $O(\dt^{p})$, as long as condition \eqref{CorCondition} is satisfied.

In the following section we will numerically explore the stabilized correction approach, and identify how different choices of $J$ may affect the stability and accuracy of the solution.

\section{Defining the stabilization matrix $\Phi$}
\label{sec:phi}
In the previous section we showed that adding a stabilization 
term to the explicit correction is expected to result in improved 
stability under certain reasonable conditions on the matrix $\Phi$.
In this section we describe different 
efficient approaches to defining and computing  $\Phi$.
These will be studied numerically in the next section.

Recall that the explicit correction strategy
\[ y^e_{[k+1]}  =  y_{exp} + \alpha \dt f(y_{[k]}) 
\]
may become unstable for large enough $\dt$.
To ensure this does not occur, we can measure
the residual at $y_{[k]}$
\[ r_{[k]} = y_{exp} + \alpha \dt f(y_{[k]}) - y_{[k]} = y^e_{[k+1]} - y_{[k]},\]
and the residual at $y^e_{[k+1]}$
\[ r^e_{[k+1]} = y_{exp} + \alpha \dt f(y^e_{[k+1]}) - y^e_{[k+1]}.\]
If the explicit correction makes the residual grow 
if $\|r^e_{[k+1]}\|_\infty \geq \|r_{[k]}\|_\infty  $ 
then we need to stabilize the explicit correction.

The stabilization approach we proposed 
\eqref{CorrStab1}  is
\[ y_{[k+1]}  =  y_{exp}  + \alpha \dt f(y_{[k]}) +
\mu \dt J  \left(y_{[k+1]} - y_{[k]} \right) ,
\]
which is
\begin{eqnarray*}
y_{[k+1]} & = & y_{[k]}
+ \Phi
\big(y_{exp}  + \alpha \dt f(y_{[k]})  - y_{[k]}
\big) \\
& = & y_{[k]} + \Phi \big( y^e_{[k+1]}  - y_{[k]} \big) \\
& = & y_{[k]} + \Phi r_{[k]}
\end{eqnarray*} 
where $\Phi = \left(I - \mu \dt J\right)^{-1}$.
Note if $\mu =0$ we recover the explicit corrections.

% {\color{red} \\
% I believe this will fix both of my issues, the definition should evaluate and subtract the current guess, as well as be lower case due to R being used in broyden section.
% \[ r^e_{[k]} = y_{exp} + \alpha \dt f(y_{[k]})
% - y^e_{[k]},\]}

% \[ R^e_{[k+1]} = y_{exp} + \alpha \dt f(y^e_{[k+1]})
% - y^e_{[k+1]},\]
% and check whether the norm of the residual has 
% increased. If it has, this indicates the need
% to stabilize the explicit correction.

We can use the same $\Phi$ over the entire simulation (a static approach) or change 
$\Phi$ at each timestep, or even at each iteration (a dynamic approach).
In the next subsections we will describe different approaches for choosing $\Phi$. 

\subsection{Static stabilization}
Ideally we can compute  $\Phi$ only once at the beginning of the 
simulation. The optimal choice of $\Phi$ will of course depend on the
best choices of $J$ and $\mu$.  We want to choose a matrix $J$
that will capture the eigenvalue spectrum of the explicit corrections,
so that $\Phi$ will then damp any growth from an explicit correction.
The choice of $\mu$ is also significant. 
In this work we use the natural choice, which  is $\mu = \alpha$.
However, in general, $\mu$ could be chosen larger to provide more stability.
This must be done with caution as modifying the size of $\mu$ will impact the error constant
and may provide less accuracy. For this reason, we use the more consistent $\mu = \alpha$,
reserving the study of different $\mu$ for future work.
% and while stabiliry is the main priority we should aim to minimize the stabilizations pollution of ther accuracy.  Otherwise it could suggest extremely large $\mu$}
We note that the implementation of static corrections
in mixed precision is straightforward;
$\Phi$ is computed in high precision, and then used for 
all the corrections.

\noindent{\bf A Jacobian-based approach:}
Our first approach involves a  static $\Phi$ based on the Jacobian 
of $f$ at the initial value:
\[ \Phi = \left(\mI - \mu \dt J_0 \right)^{-1} 
\; \; \; \mbox{where} \; \;  J_0 = f'(y_0). \] 
% The value of $\mu$ can be chosen to optimize stability;
% in our case we generally use $\mu = \alpha$.

\smallskip

\noindent{\bf  An approach based on the differential operator:}
The Jacobian approach is tied to the initial value; an alternative
is to consider the dominant differential operator of $f$, 
and use it as a basis for $\Phi$.
In our case, we say
\[ y_{[k+1]} = y_{exp} + \alpha \dt f(y_{[k]}) + \dt \mu \mathcal{L} 
(y_{[k+1]} - y_{[k]}),\]
where $\mathcal{L}$ is an approximation of a spatial derivative.

The Burgers' example and the shallow water equations
both have the derivative operator applied to a function, so 
we would choose the derivative operator as $\mathcal{L}$
\[ f(u) = - D_x \big(\mathcal{F}(u) \big)  \; \;  
\implies \; \;  \mathcal{L} = - D_x \] 
(in the shallow water system this would be more properly defined as $\mathcal{L} = diag(D_x,D_x)$. 
For the nonlinear heat equation we have
\[ f(u) = D_{xx} \big(u^m \big)  \; \;  
\implies \; \;  \mathcal{L} = D_{xx} .\] 
% The appropriate choice of $\mu$ may depend on the problem 
% and the time-stepping method.
This approach is inspired by the Explicit-Implicit-Null (EIN) 
method, which consists of adding and subtracting a derivative operator that mimics the spatial dynamics, multiplied by a scaling parameter $\mu$,
and then developing an IMEX method based on this decomposition.

\subsection{Dynamic stabilizations}
If the time-step is refined, $\Phi$ would likely need
to be modified as well, as the time-step refinement is similar
to modifying $\mu$. Allowing $\Phi = \Phi_k$ may be advantageous in different cases,
at the cost of added expense. Some approaches involve using a diagonal,
tridiagonal, or triangular matrix which adds little cost but can be 
efficiently solved at each iteration.

Another approach is to update the matrix $\Phi_k =(I - \alpha \dt J_k)^{-1}$
without the  need to compute an implicit solve at every time-step,
using a Broyden-type approach \cite{broyden} to the update.
We use this approach to inexpensively compute an increment 
$\Delta \Phi_k$ such that
\[\Phi_{k}= \Phi_{k-1} + \Delta \Phi_k ,\]
where $\Phi_{k-1}$ has been previously computed.
To accomplish this, we  write the corrections in the  form
\[  y_{[k+1]}  = y_{[k]} -  \Phi_k  F(y_{[k]}) . \]

If we want $\Phi_k$ to satisfy a secant-type condition
\[ \Phi_k  \big(F(y_{[k]}) - F(y_{[k-1]}) \big) 
= y_{[k]} - y_{[k-1]},  \]
and noting that $\Phi_k = \Phi_{k-1} + \Delta \Phi_k$,
we have 
\begin{eqnarray*}
\big(\Phi_{k-1} + \Delta \Phi_k \big)
\big(F(y_{[k]}) - F(y_{[k-1]}) \big) 
= y_{[k]} - y_{[k-1]}.
\end{eqnarray*}  
Let 
\[ R_k = F(y_{[k]}) - F(y_{[k-1]}) =\left(y_{[k]} -y_{[k-1]} \right)  -
\alpha \dt \left( f(y_{[k]}) - f(y_{[k-1]}) \right) ,\]
and
\[ \Upsilon = y_{[k]} - y_{[k-1]} - \Phi_{k-1} R_k\]
and we  wish to solve
\begin{eqnarray}
    \Delta \Phi_k \; R_k = \Upsilon_k.
\end{eqnarray} 
Here, $\Phi_k$ is known, and $R_k$ and $\Upsilon_k$ are
based on pre-computed values.
This problem has infinitely many possible rank one solutions,
which can be found by using any vector $\rho$ so that 
$\rho^T R_k \ne 0$ and setting
\[ \Delta \Phi_k = \frac{1}{\rho^T R_k} \Upsilon_k \rho^T. \]

The secant-type algorithm proposed by Broyden gives
two possibilities for this vector $\rho$, commonly known as "bad Broyden's" and "good Broyden's". These names are not always indicative of their performance. 
We can select $\rho = R_k$, which is not zero
unless we have $R_k =0$, in which case we have already converged
and need not correct any further.
A second approach would be to take $\rho^T = \Upsilon_k^T \Phi_{k-1}$.
We describe these two approaches in the algorithm below:
\begin{framed}
    {\bf Broyden's Algorithm}
\begin{enumerate}
    \item At each stage of the time-step, we compute
    $y^{(i)}_{[0]}$ using the inexpensive low-accuracy
    solve of
    \[ y^{(i)}_{[0]} = y_n + \dt \sum_{j=1}^{i-1} a_{ij} f(y^{(j)})
    + \dt a_{ii} f_\varepsilon ( y^{(i)}_{[0]}) . \]
    \item We begin with the precomputed value 
    $\Phi_{0}$. \\
    {\em (Note that we will often use the initial 
    $\Phi_{k-1} = (I-\alpha \dt J_0)^{-1}$, 
    where $J_0$ is the initial time Jacobian 
    $J_0 = f'(y_0)$, or the differential operator $\mathcal{L}$
    described above. In this sense we are updating the 
    frozen Jacobian approach described above).}
    \item  Evaluate the increment $\Delta \Phi_k$ by:
    \begin{enumerate}
        \item Calculate 
        $ R_k =   F(y_{[k]}) - F(y_{[k-1]})$ 
        where $ F(y_{[k]}) = y_{[k]} - y_{exp} - \alpha \dt  f_{[k]}   . $
        \item Use this to evaluate
        $\Upsilon_k =   y_{[k]} - y_{[k-1]} -  \Phi_{k-1} R_k  .     $
        \item Now compute 
        \[\Delta \Phi_k = \frac{1}{\rho^T R_k} \Upsilon_k \rho^T,
        \; \;  \mbox{where} \; \; 
        \left\{ 
        \begin{array}{ll}
        \rho = R_k, &  \mbox{``Bad'' Broyden's}  \\
         \rho = \Upsilon_k^T \Phi_{k-1} &   \mbox{``Good'' Broyden's}
         \end{array}  \right. 
         \]    
       { \em (Note that the Broyden update can be done at each time-step, at each stage, or even at each corrections. We found the once per time-step works best of these three options).}
    \end{enumerate}
    \item Compute $\Phi_k = \Phi_{k-1} + \Delta \Phi_k$.
   \item The correction (for $k=0:p-2$) takes the form
    \[ y_{[k+1]} = y_{[k]}  -   \Phi_k  F(y_{[k]}) ,\]
    where \[ F(y_{[k]}) = y_{[k]} - y_{exp} -   \alpha \dt  f_{[k]}   . \]
\end{enumerate}
\end{framed}

\section{Numerical Results}
\label{sec:NumCor}

% \begin{verbatim}
%     linearization -- no pertubation 
 
% Burger’s 
% Nx = 50-100-150-200-250 
% dt = 1e-1 : 1e-4   8 dt’s 
% Ic: ½ + ¼sin(x) 
%  invid t_final 0.5 
%  Vis = T_final = 2,5,     mu = 0.00001 
% Record: for SDIRK2-4 
%     z_i – y_i   save big matrix time and stage for each, for all time put Nx on stages, 
%     h_i = f_i – f_i_lin    save big matrix time and stage for each, for all time put Nx on stages, 
%    Error = u_final – u_ref   only at T_final 
% Put Error on separate plot for all methods vary NX, same X and Y axis 
% Similar for non heat 
 
% linearization – Perturbation 1e-2, 1e-4,1e-8 
% Same as above 
% linearization – Perturbation 0, 1e-2, 1e-4,1e-8 – w corrections 
% Only record Error = u_final – u_ref 
% # of corrections  := 1:p  for pth order SDIRK 
% EIN - M_0 - Both with alpha [0,0.5,1,2] 
% M_0 – fixed with good and bad broyden 
% Correct once per time-step, one per stage, and once per correction per stage. 
% Same for non-linear heat. Same for newtons with chopped digits RHS. 
% Try to match colors and line patterns of Cesar. Each method has its color, each # of corrections has its own marker/line time 
 
%\end{verbatim}

\subsection{Inviscid Burgers' equation}

\subsubsection{Linearization and perturbation}
We begin with the inviscid Burgers'  \eqref{eq:Burgers}
with initial condition $u(x,0)= \frac{1}{2} + \frac{1}{4} \sin(x)$.
We are interested in the solution of this equation at time $T_f = 3.5$, which is before the shock forms.

As above, the  linearization is performed using a Taylor series:
\begin{eqnarray}
f_{\epsilon}(y) &=&  f(\bar{y}) + 
    f'(\bar{y}) \left( y - \bar{y} \right) 
    =  - \frac{1}{2} D_x \bar{y}^2
    - D_x \bar{Y} \left( y - \bar{y} \right) .
\end{eqnarray}
We typically linearize using $\bar{y} = u^n$. 
In addition, at the implicit solve we  perturb the matrix 
$(I - a_{ii} \dt D_x \bar{Y})^{-1}$
by chopping it off after a set number of digits $d$, 
leading to  a perturbation of $\epsilon_{pert} = 10^{-d}. $
This allows us to account for additional errors, such as those resulting from a 
less accurate linear solver, in addition to the linearization error.

We test the different correction strategies: 
the explicit correction,  and the stabilized correction approaches proposed  
in Section \ref{sec:phi}, with three different approaches to computing $\Phi$:
%\[ \Phi = \big( I + \mu \dt J \big)^{-1} \]
\begin{enumerate}
    \item  Static $\Phi_J = \Phi_0 = \big(I - \mu \dt J_0 \big)^{-1}$
    based on a frozen Jacobian  $J_0 = f'(y_0)$, Plotted in blue.
    \item Static  $\Phi_{EIN} = \big(I + \mu \dt D_x \big)^{-1} $, where $J = - D_x$ is based on the EIN
    approach. 
    plotted in magenta.
    \item Dynamic $\Phi_B$ using ``Bad Broyden's'' algorithm plotted in cyan. 
    We generally find the ``Bad Broyden's'' to work better than the ``Good Broyden's''. 
\end{enumerate}
We chose to let $\mu = a_{ii}$ to best match with the underlying scheme.
%however further investigation may be done to study the impact on stability of the method for various choices of $\mu$.

\begin{figure}[t]
\centering
\begin{subfigure}[b]{\textwidth}
\centering
\includegraphics[width=0.31\textwidth]{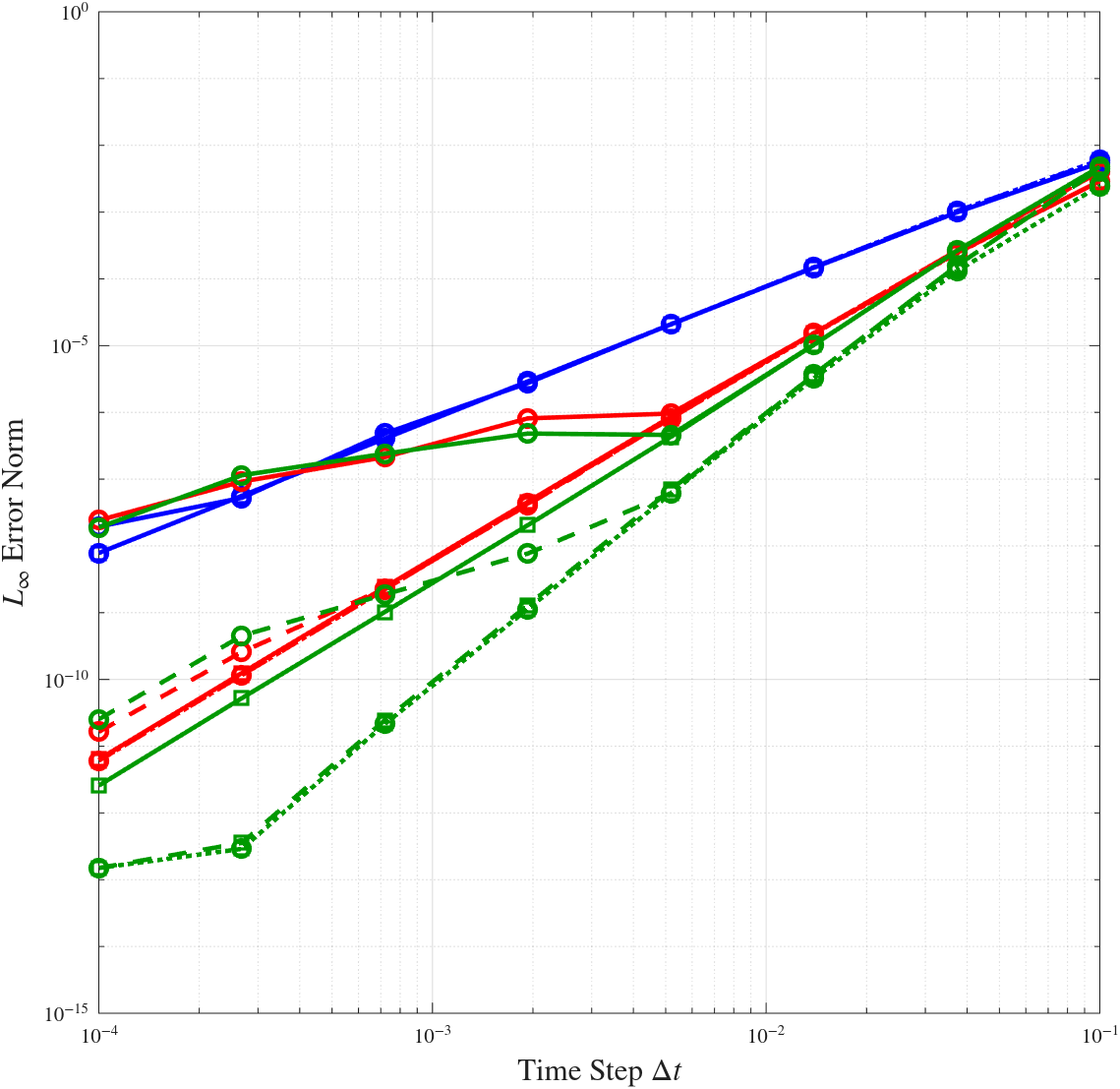}
\includegraphics[width=0.31\textwidth]{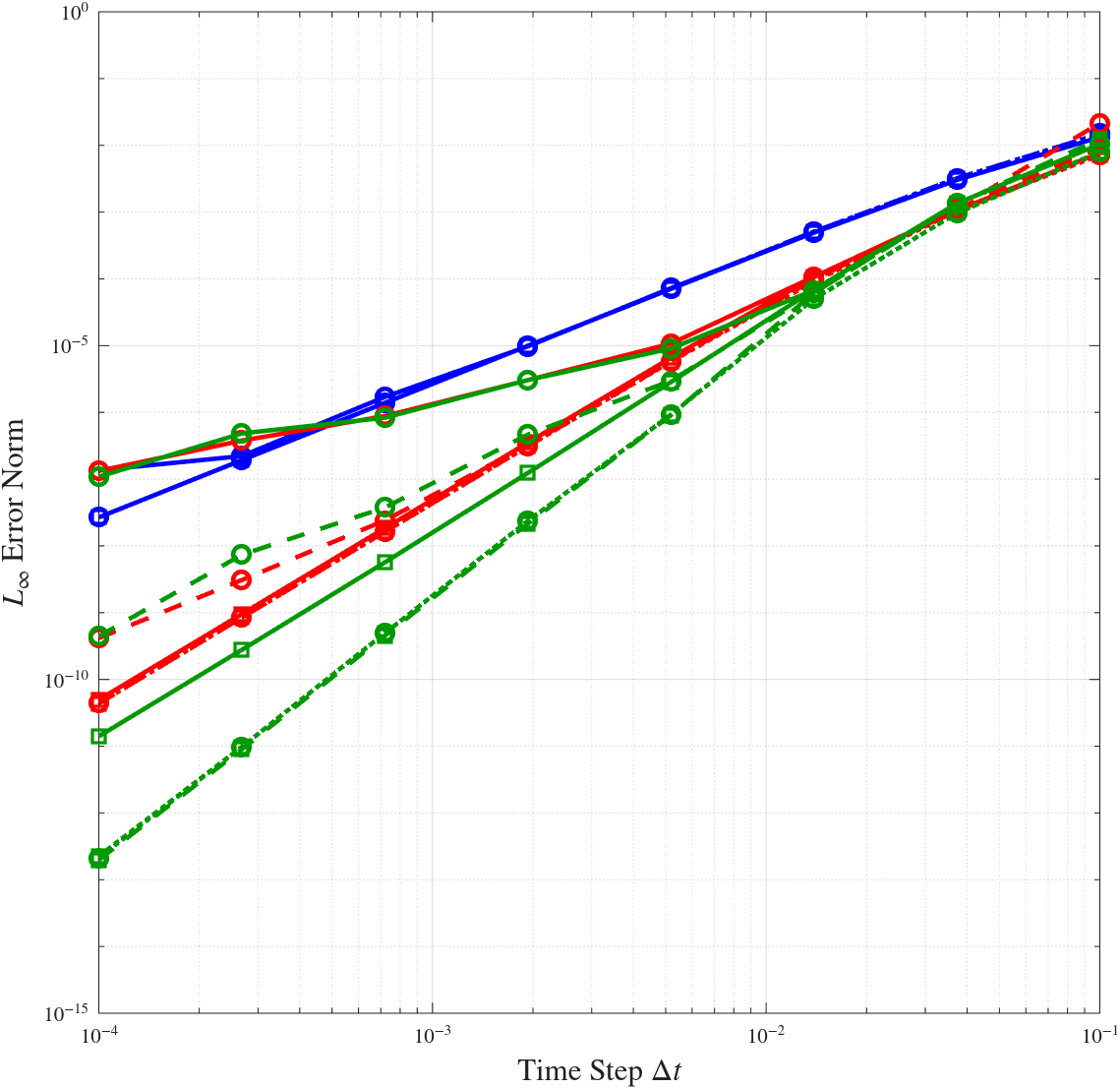}
\includegraphics[width=0.31\textwidth]{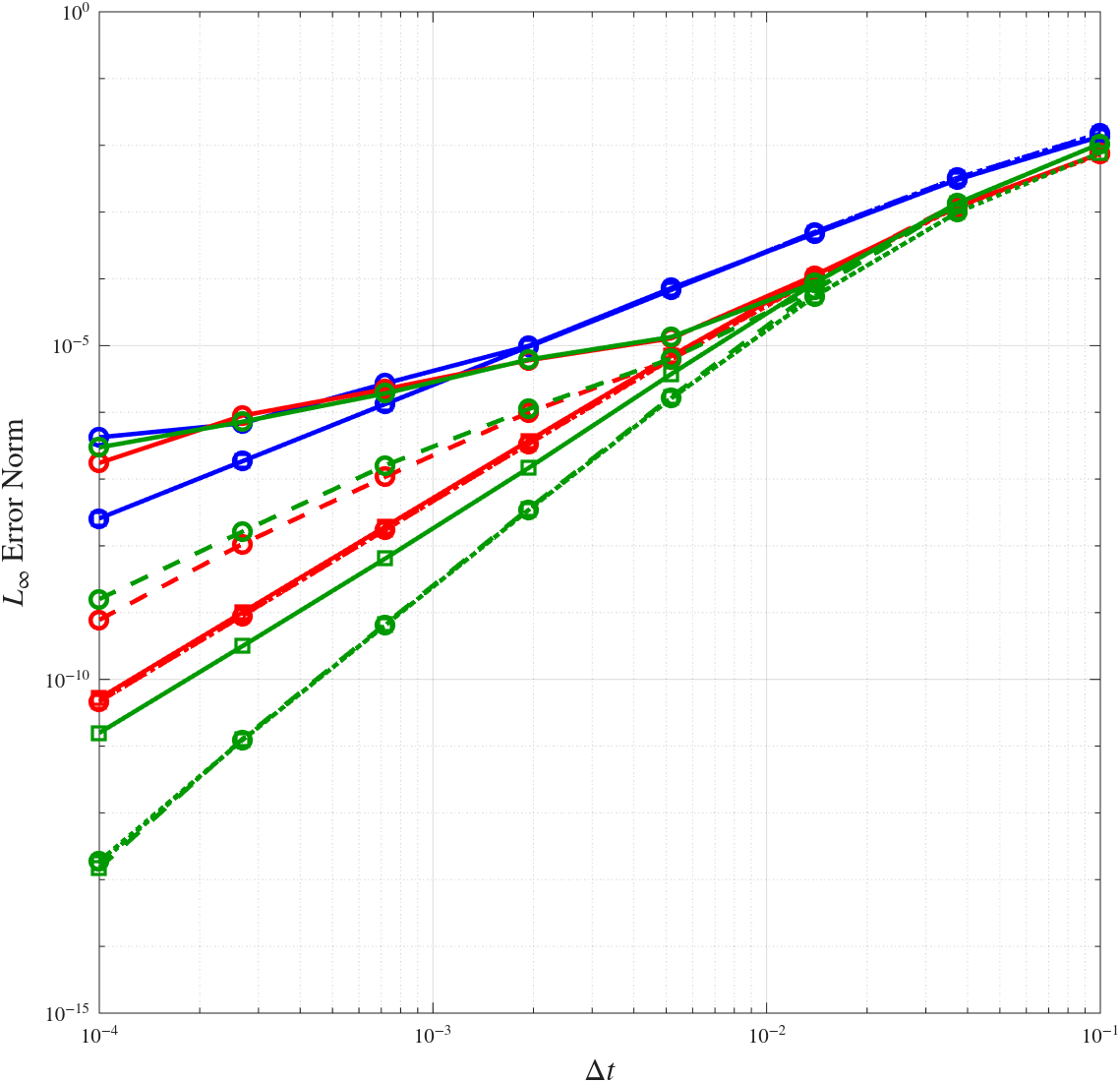}
\end{subfigure} 
\caption{Correction approaches for the time-evolution of the inviscid Burgers'
with $N_x=50$ (left), $N_x=250$ (middle), and $N_x=450$ (right).
We evolve this to final time $T_f= 3.5$ using 
    SDIRK2 in blue, SDIRK3 in red,
    and SDIRK4 in green.
    We use no corrections (solid lines),
    explicit corrections (dashed line), and 
    the static stabilized correction $\Phi_J$  (dotted line). 
    % Solid lines: no correction; dashed line: explicit correction;
    % dotted line stabilized corrections.
The method with no perturbation has square markers, the 
perturbed method with $\epsilon = 10^{-4}$ has round markers.
    \label{fig:LinBurgersCor}}
\end{figure}

In Figures \ref{fig:LinBurgersCor} we show the impact of the 
different stabilized correction strategies on the solution using 
the second order SDIRK2 (left), the third order
SDIRK3 (middle), and the fourth order SDIRK4 (right),
we use $p-1$ corrections for a $p$th order method.
We show the evolution using  $N_x=50$ points in space (left),
$N_x =250$  points (middle), and $N_x =450$  points (right). 
The method with no perturbation has square markers, the 
perturbed method with $\epsilon = 10^{-4}$ has round markers.
No corrections are solid lines, explicit corrections are dashed,
and $\Phi_J$ are dotted lines.
The corrections here are all stable, and correct the impact of the 
linearization and perturbation. However, as $N_x$ gets larger and we have a perturbation,
we find these corrections are not as impactful. Perhaps more
corrections could be beneficial for such cases.
We note that the $\Phi_{EIN}$ and $\Phi_B$  stabilized corrections  perform the same, 
so are not shown in this figure.

\subsubsection{Mixed precision implementation}
 Once again we begin with the inviscid Burgers' equation \eqref{eq:Burgers}
but  here we use the   initial condition $u(x,0)=  \sin(x)$. 
 We semi-discretize using a spectral differentiation matrix with $N_x$ points
 and step these forward to final time $T_f =0.7$
 using the mixed precision SDIRK2 \eqref{MPIMR}, SDIRK3 \eqref{MPSDIRK3}, 
 and SDIRK4 \eqref{MPSDIRK4}.

We start by computing the explicit correction
 \[ y^e_{[k+1]} =  y_{[k]}  + \alpha \dt f(y_{[k]})\]
and we stabilize with a high precision stabilization matrix $\Phi$
 \[ y_{[k+1]} =  y_{[k]}  + \Phi \left(y_{[k+1]}^e -y_{[k]} \right). \]
This high precision matrices $\Phi = \Phi_J$  and $\Phi =  \Phi_{EIN}$
are computed only once, at the initial time-step.

\begin{figure}[hbt]
     \centering
    % --- IMR Plots ---
    \begin{subfigure}[b]{\textwidth}
        \centering
\includegraphics[width=0.48\textwidth]{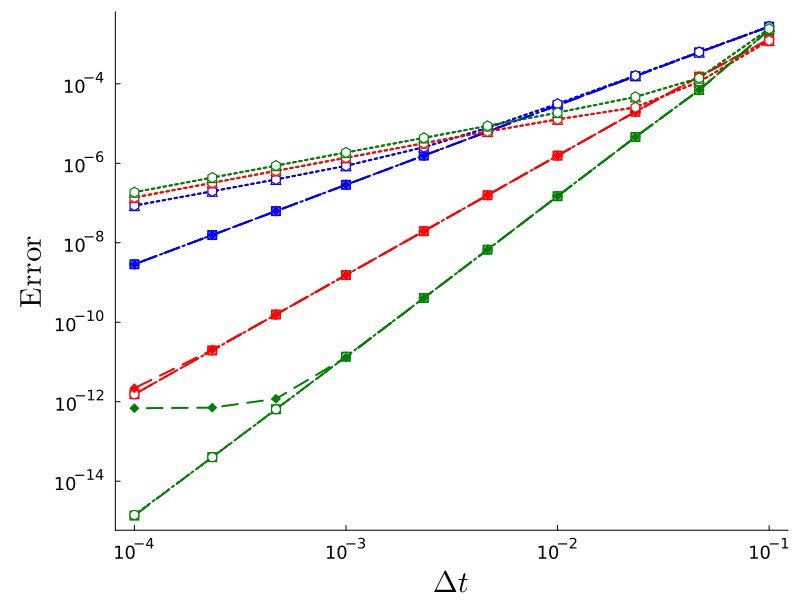}
\includegraphics[width=0.48\textwidth]{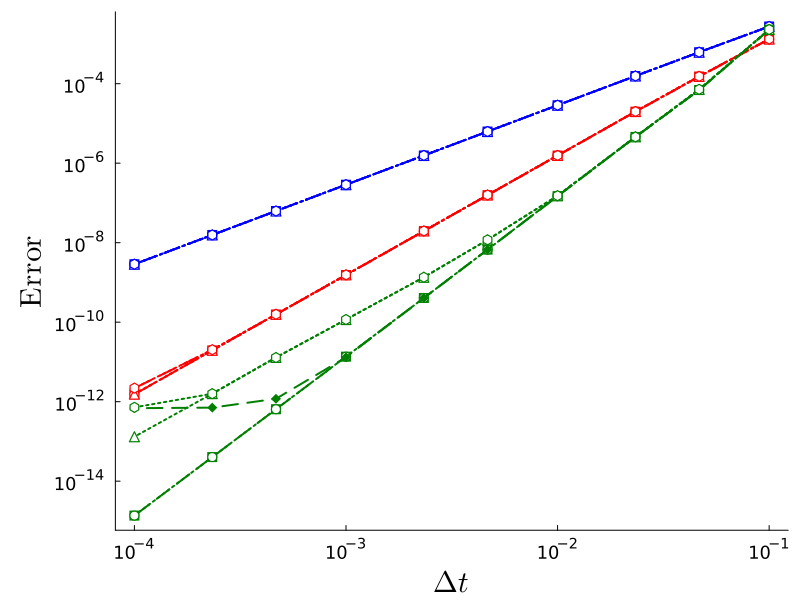}
    \end{subfigure} \\
    \begin{subfigure}[b]{\textwidth}
        \centering
\includegraphics[width=0.48\textwidth]{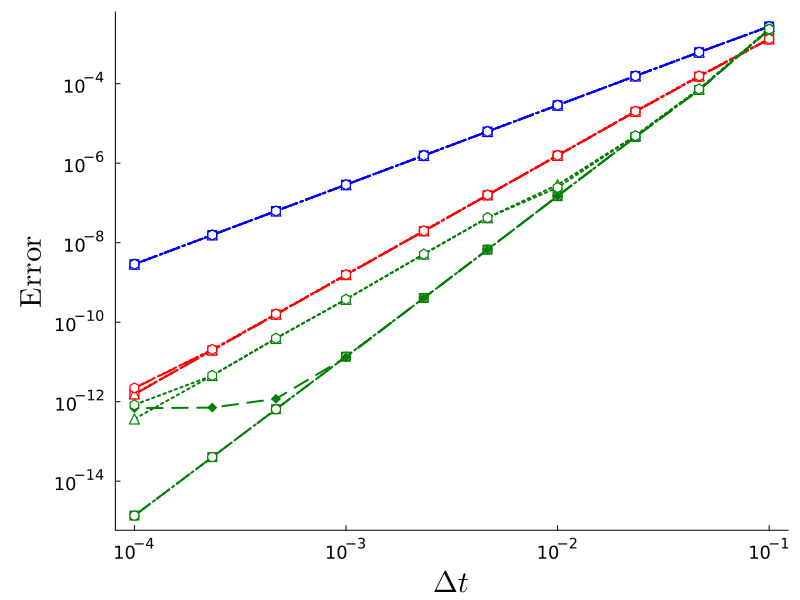}
\includegraphics[width=0.48\textwidth]{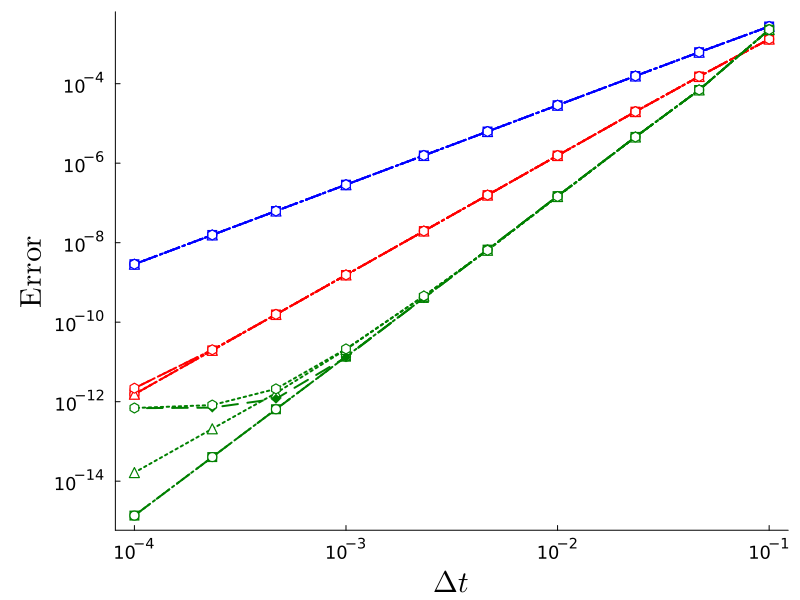}
    \end{subfigure}
    \begin{minipage}[c]{0.3\textwidth}  \hspace{0.1in}
    \includegraphics[width=0.85\textwidth]{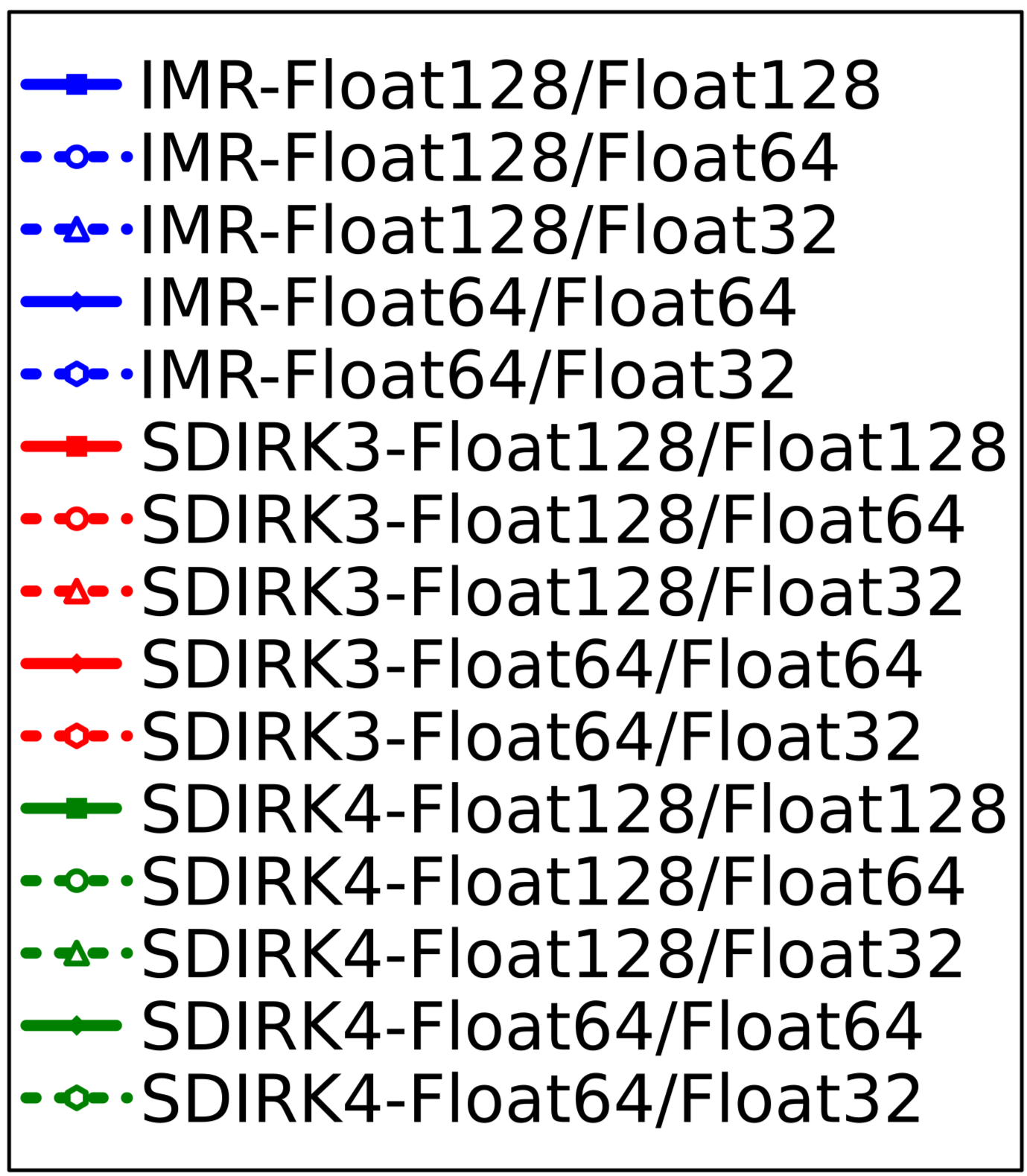}
\end{minipage} \hfill
    \begin{minipage}[c]{0.65\textwidth} 
\caption{Mixed precision Burgers' equation with
$N_x = 200$ spatial points.  
SDIRK2 \eqref{MPIMR} in blue;  
SDIRK3  \eqref{MPSDIRK3} in red; 
SDIRK4 \eqref{MPSDIRK4} in green. \\
Top Left: No corrections. 
Top Right: two explicit corrections. 
Bottom Left:  two corrections with $\Phi_{EIN}$ stabilization.  
Bottom Right: two corrections with $\Phi_J$ stabilization.
%{\bf does the legend have solid lines that are not in the figure?}
    \label{fig:BurgersMP}} 
    \end{minipage} \hfill 
\end{figure}

Figure \ref{fig:BurgersMP} shows the impact of two corrections for 
$N_x=200$, for SDIRK2 \eqref{MPIMR} (blue), SDIRK3 \eqref{MPSDIRK3} (red),
and SDIRK4 \eqref{MPSDIRK4} (green) methods.
On the top left is the mixed precision implementation without corrections. 
On the top right we see the explicit corrections;
for this case the explicit corrections do not cause the method to become unstable.
This means that the impact of the stabilization may not be evident. Indeed,
for the $\Phi_{EIN}$ stabilization (bottom left) there is no improvement over
explicit corrections, and in some cases it is even worse. For the
$\Phi_J$ based stabilization (bottom right) there is significant improvement 
in the accuracy, particularly in the fourth order method SDIRK4.

\subsection{Shallow water equations}
In this section we study the impact of linearizations and mixed precision,
with and without corrections, on a system of equations.  
Consider the shallow water equations:
\begin{align} \label{sw1}
\eta_t + (\eta u)_x = 0, \; \; \; \;  & \; \;  \; \; 
(\eta u)_t + \left( \eta u^2 + \frac{1}{2}\eta^2 \right)_x = 0,
\end{align}
for $x \in (0, 2\pi)$, with initial conditions
$\eta(x, 0) = 0.1\times \sin(x) +1$, $u(x, 0) = 0$, and periodic boundary conditions. Here $\eta(x, t)$ denotes the height and $u(x, t)$ the velocity.  
Let $\mu = \eta u$ be the mass flux, then \eqref{sw1} can be written as 
\[ \eta_t + \mu_x = 0, \; \; \; \;
\mu_t  + \left( \frac{\mu^2}{h} + \frac{1}{2}\eta^2 \right)_x = 0. \]
Once again we semi-discretize this system of equations using a Fourier 
spectral method differentiation matrix $D_x$, and the function $f(y)$ is given by 
\[ y' = \begin{pmatrix} y_{\eta}' \\
                y'_{\mu} \end{pmatrix} = f(y) = -\begin{pmatrix}
                    D_xy_\mu\\D_x\left[\frac{y_\mu^2}{y_\eta} + \frac{1}{2}y_\eta^2\right]
                \end{pmatrix}. \]

\subsubsection{Linearizations}
We linearize using a Taylor expansion  around
$\bar{y} = y_n$:
\begin{eqnarray} \label{linSW3}
f_\varepsilon(y) &= & f(\bar{y}) + f'(\bar{y})(y - \bar{y}) = -\begin{pmatrix}
                    D_x\bar{y}_\mu\\D_x\left[\frac{\bar{y}_\mu^2}{\bar{y}_\eta} + \frac{1}{2}\bar{y}_\eta^2\right]
                \end{pmatrix} + f'(\bar{y})\begin{pmatrix}
                    y_\eta - \bar{y}_\eta\\ y_\mu - \bar{y_\mu}
                \end{pmatrix},
       \end{eqnarray}
     where  
     \[  \bar{Y_\eta} = \text{diag}(\bar{y_\eta}) \quad \mbox{and} \quad \bar{Y}_{\mu/\eta} = \text{diag}(\bar{y}_\mu/\bar{y_\eta}). \] 
and 
     \begin{equation*}
         f'(\bar{y}) = \begin{pmatrix}
             {\bf 0} & -D_x\\D_x \left[\left(\bar{Y}_{\mu/\eta} \right)^2  - \bar{Y_\eta}  \right] &-2D_x \bar{Y}_{\mu/\eta}
         \end{pmatrix}.
     \end{equation*}

\begin{figure}[b!]
     \centering
    \begin{subfigure}[b]{\textwidth}
        \centering
        \includegraphics[width=0.49\textwidth]{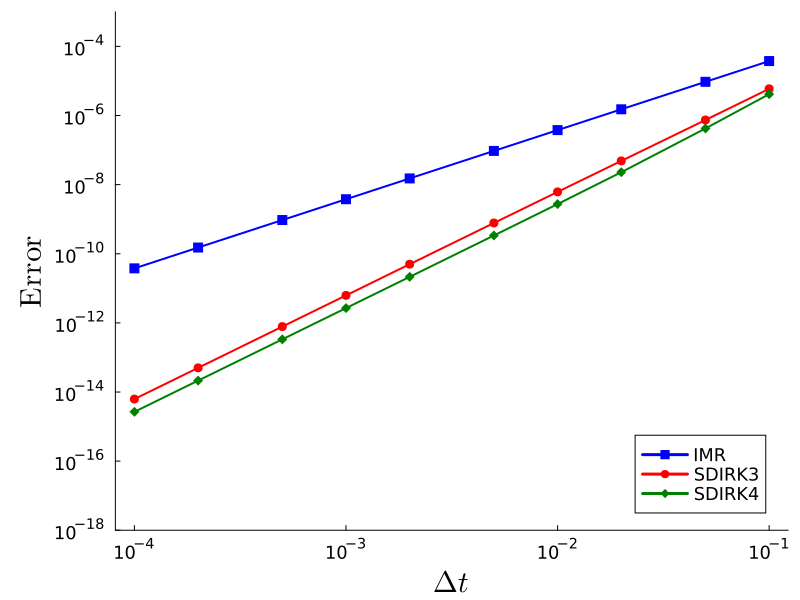}
        \includegraphics[width=0.49\textwidth]{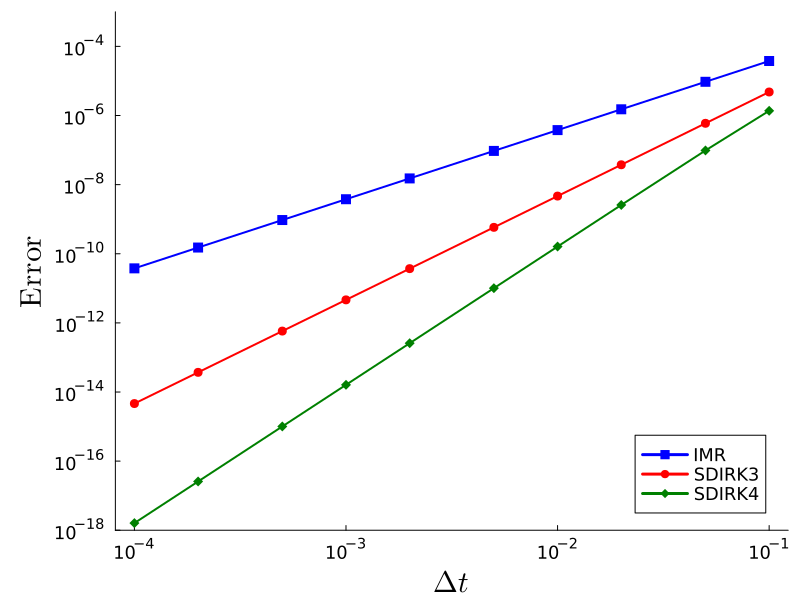}
    \end{subfigure} \\
    % \begin{subfigure}[b]{\textwidth}
    %     \centering
    %     \includegraphics[width=0.49\textwidth]{Plots/Linearization/SW/Figure_AllMethods_Quad_Nx100_2_corr_2.png}
    %     \includegraphics[width=0.49\textwidth]{Plots/Linearization/SW/Figure_AllMethods_Quad_Nx100_2_corr_1.png}
    %\end{subfigure}
\caption{Shallow water equations  with a Taylor series linearization for $N_x = 100$.\\
Top Left: No corrections. 
Top Right: two explicit corrections.  
% Bottom Left:  two EIN-based stabilized corrections.  
% Bottom Right: two frozen Jacobian based stabilized corrections.
    \label{fig:SWlinFT}} 
\end{figure}

In Figure \eqref{fig:SWlinFT} we show the impact of the different corrections on
the linearized shallow water equations. On the top left we have no corrections.
On the top right we see the impact of two explicit corrections. We note that these are stable 
for all tested values of $\dt$. The $\Phi_{EIN}$ and $\Phi_J$ 
stabilized corrections perform similarly,  and they all correct the accuracy 
of the methods to the design order.
This is not shown in the figure, as the three graphs look identical.

\subsubsection{ Mixed precision implementation}
Using the mixed precision algorithm described above
we evolve this to final time $T_f = 0.5$ using the mixed precision SDIRK2 \eqref{MPIMR},
SDIRK3 \eqref{MPSDIRK3}, and SDIRK4 \eqref{MPSDIRK4}.
The precisions we use are quad mixed with double and single, and double mixed with single.
In Figure \eqref{fig:SWMPFT} we look at the final time { maximum norm} errors resulting from 
a mixed precision implementation of the shallow water equations with $N_x = 100$ spatial points
(top left), when compared to a reference solution.
We compare these to the  two explicit corrections (top right),
\begin{figure}[t!]
     \centering
    \begin{subfigure}[b]{\textwidth}
        \centering
        \includegraphics[width=0.49\textwidth]{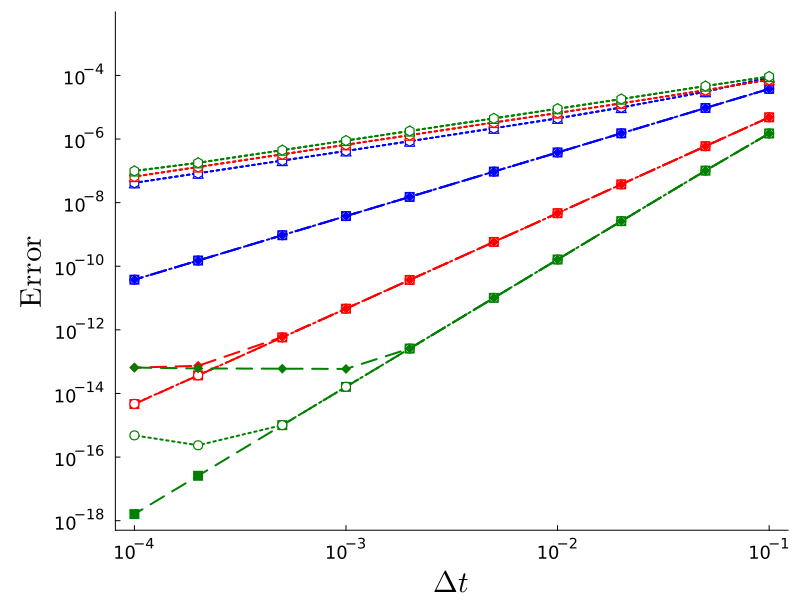}
        \includegraphics[width=0.49\textwidth]{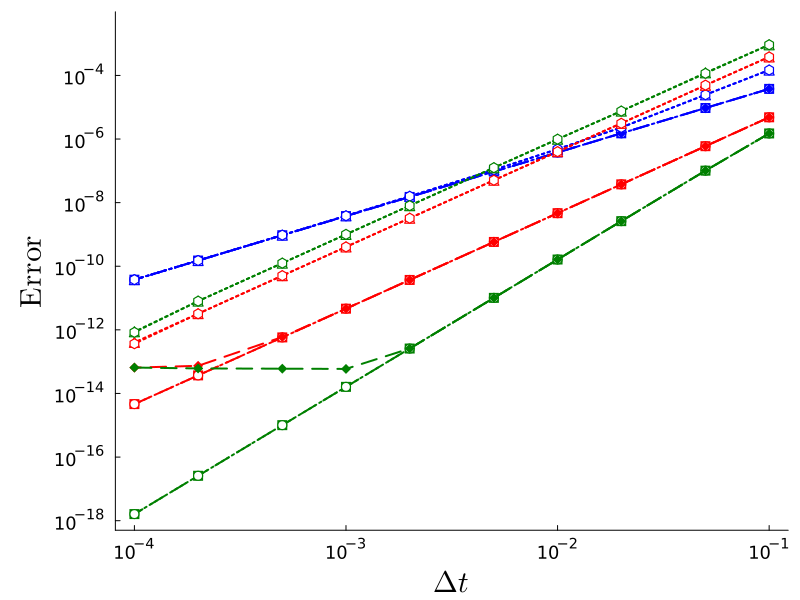}
    \end{subfigure} \\
    \begin{subfigure}[b]{\textwidth}
        \centering
        \includegraphics[width=0.49\textwidth]{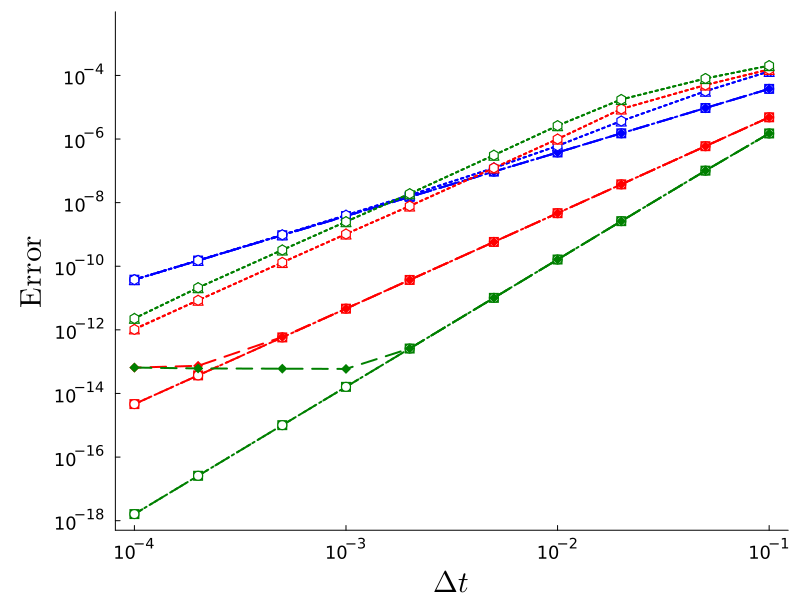}
        \includegraphics[width=0.49\textwidth]{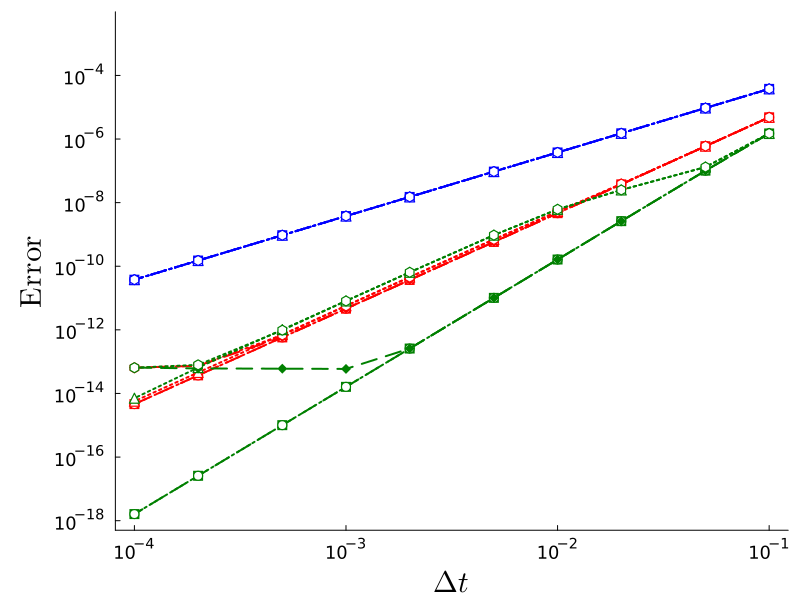}
    \end{subfigure}
    \begin{minipage}[c]{0.3\textwidth}  \hspace{0.1in}
    \includegraphics[width=0.85\textwidth]{Plots/MP_Burgers/Figure_AllMethods_Quad_Legend.png}
\end{minipage} \hfill
    \begin{minipage}[c]{0.65\textwidth} 
\caption{Mixed precision shallow water equations \\
$N_x = 100$ spatial points. \\
SDIRK2 \eqref{MPIMR} in blue;  
SDIRK3 \eqref{MPSDIRK3} in red; 
SDIRK4 \eqref{MPSDIRK4} in green. \\
Top Left: No corrections. \\
Top Right: Two explicit corrections. \\
Bottom Left:  Two corrections with  $\Phi_{EIN}$ stabilization.  \\
Bottom Right: Two corrections with  $\Phi_J$ stabilization.
%{\bf does the legend have solid lines that are not in the figure?}
    \label{fig:SWMPFT}} 
    \end{minipage} \hfill 
\end{figure}
two $\Phi_{EIN}$ stabilized corrections (bottom left),
and two $\Phi_J$ stabilized corrections  (bottom right).
The explicit corrections are stable for all the values of $\dt$ tested, and they do a very
good job improving on the accuracy of the mixed precision method. The $\Phi_{EIN}$ 
stabilized corrections perform similarly to the explicit corrections. 
The bottom right figure shows that the $\Phi_J$  stabilized corrections
outperform the other corrections in terms of the improvements in accuracy.
In these cases we observed that the $\Phi_J$ stabilized corrections  
are not only stabilizing, but provides accuracy advantages as well. 
This becomes more evident in the next example.

\subsection{Porous medium problem}

Our final example is the nonlinear equation 
 \begin{eqnarray}  \label{eq:PM}
 u_t  = (u^3)_{xx},
 \end{eqnarray}
 on the domain $x = (-\pi,\pi)$, 
 with initial condition $u(x,0)= \frac{1}{2} \cos(x)+  \frac{1}{2} $ 
 and periodic boundary conditions. Once again we use a spectral differentiation matrix
 for the spatial discretization, and evolve the resulting ODE system
 using the three mixed accuracy time-stepping methods
SDIRK2 \eqref{MPIMR}, SDIRK3 \eqref{MPSDIRK3}, and SDIRK4 \eqref{MPSDIRK4}
to a final time $T_f = 0.5$.

\begin{figure}[b!]
\centering
\begin{subfigure}[b]{\textwidth}
\centering
\includegraphics[width=0.48\textwidth]{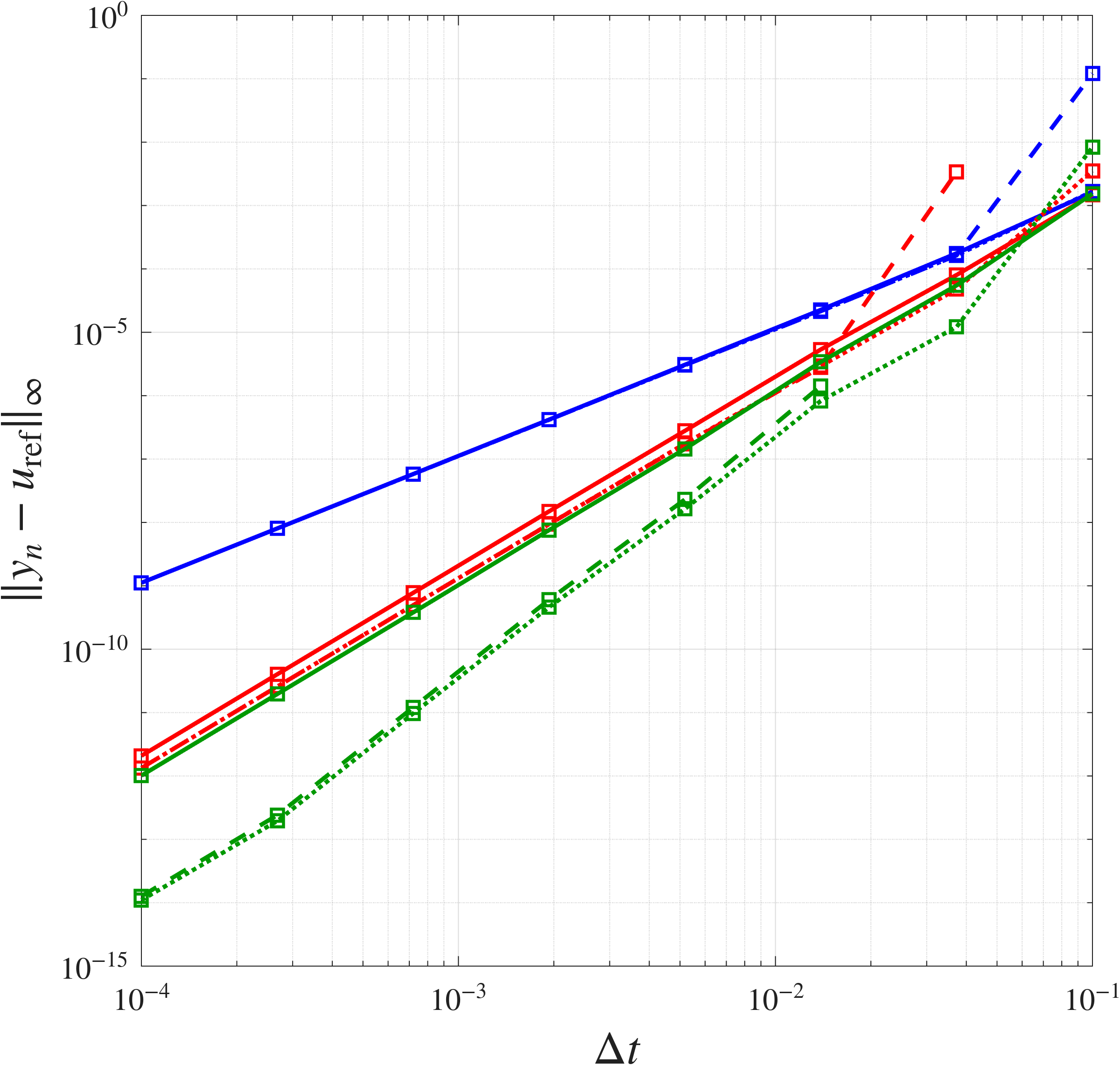}
\includegraphics[width=0.48\textwidth]{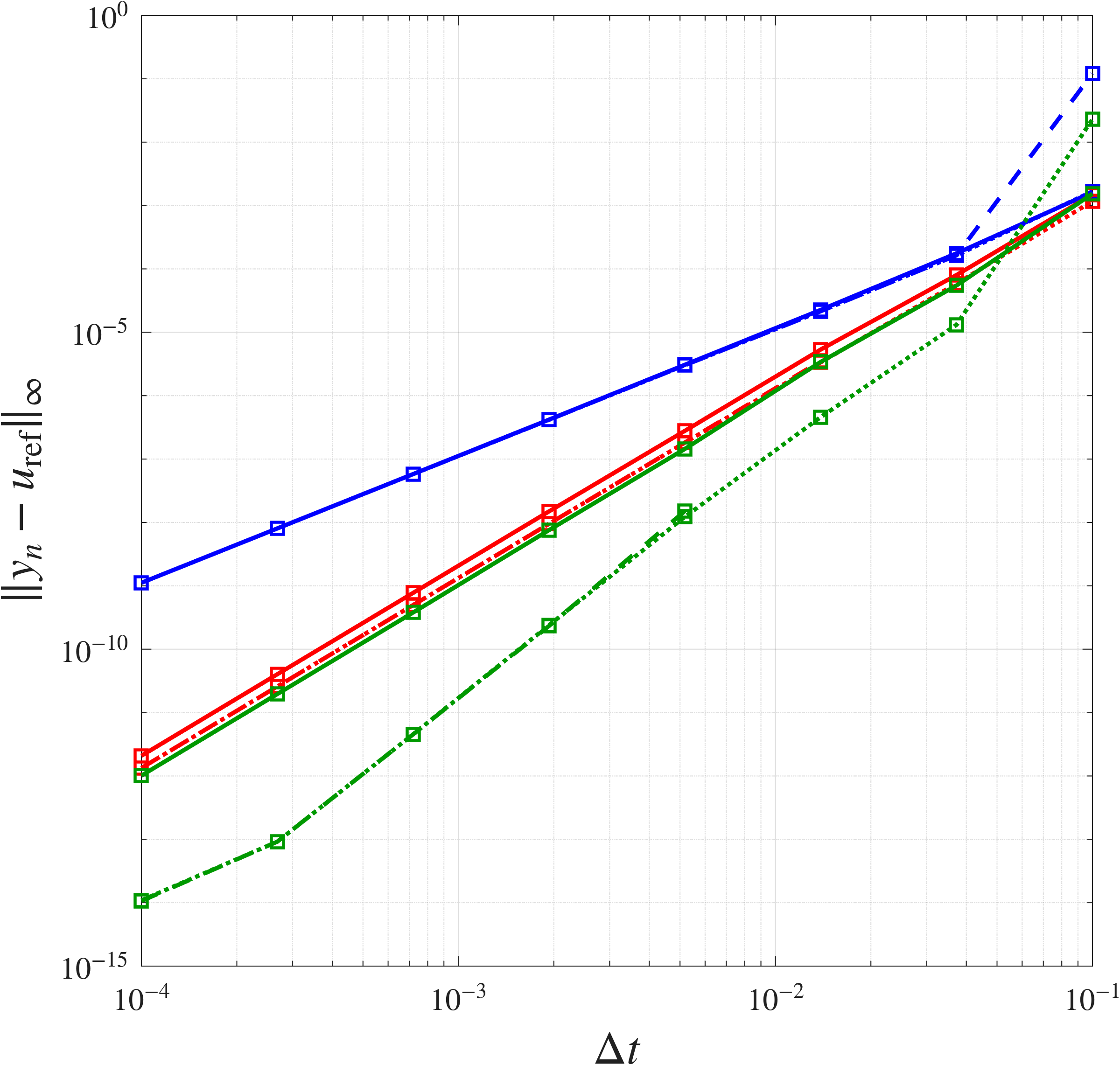}
\end{subfigure} \\
\begin{subfigure}[b]{\textwidth}
\centering
\includegraphics[width=0.48\textwidth]{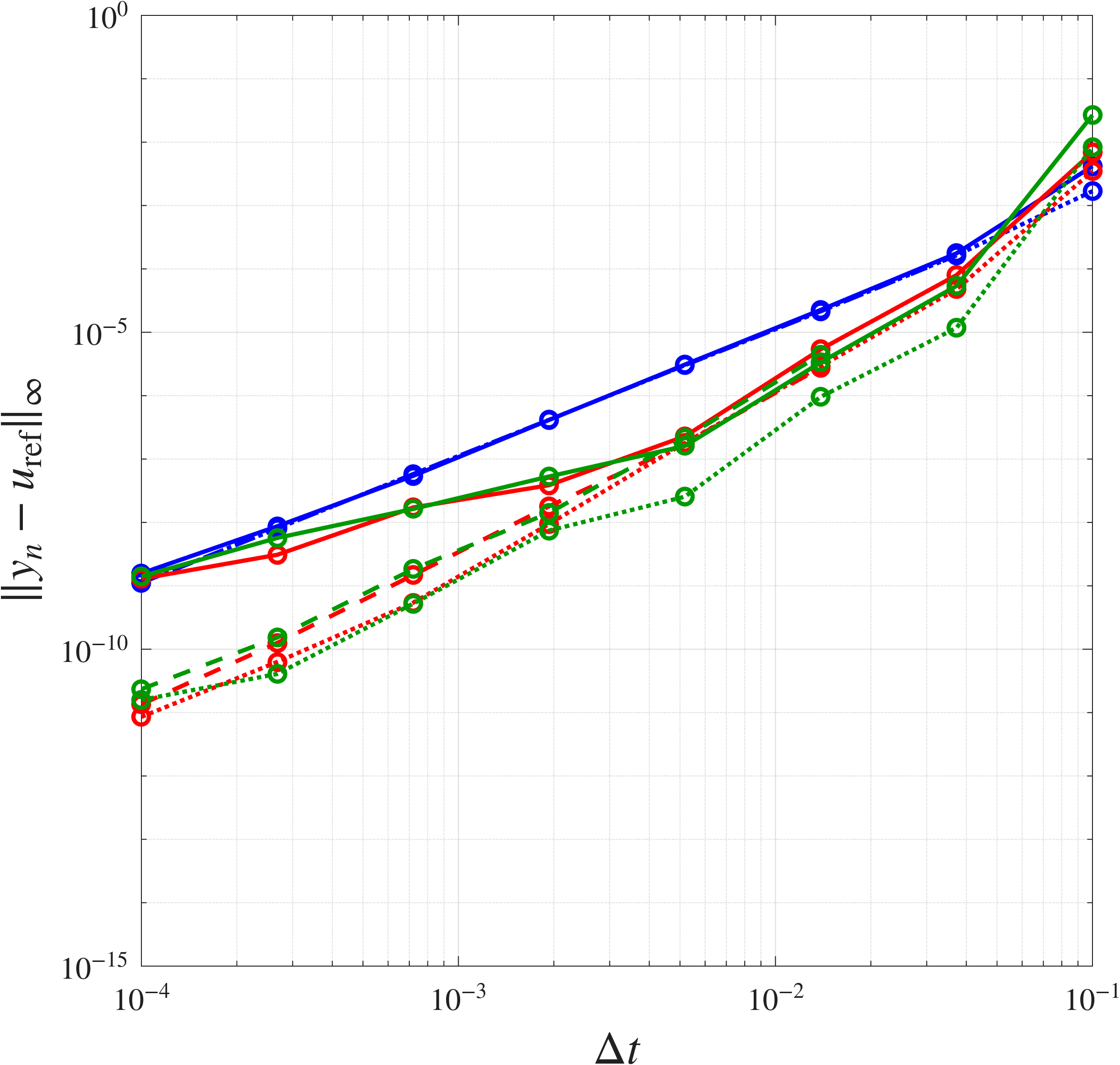}
\includegraphics[width=0.48\textwidth]{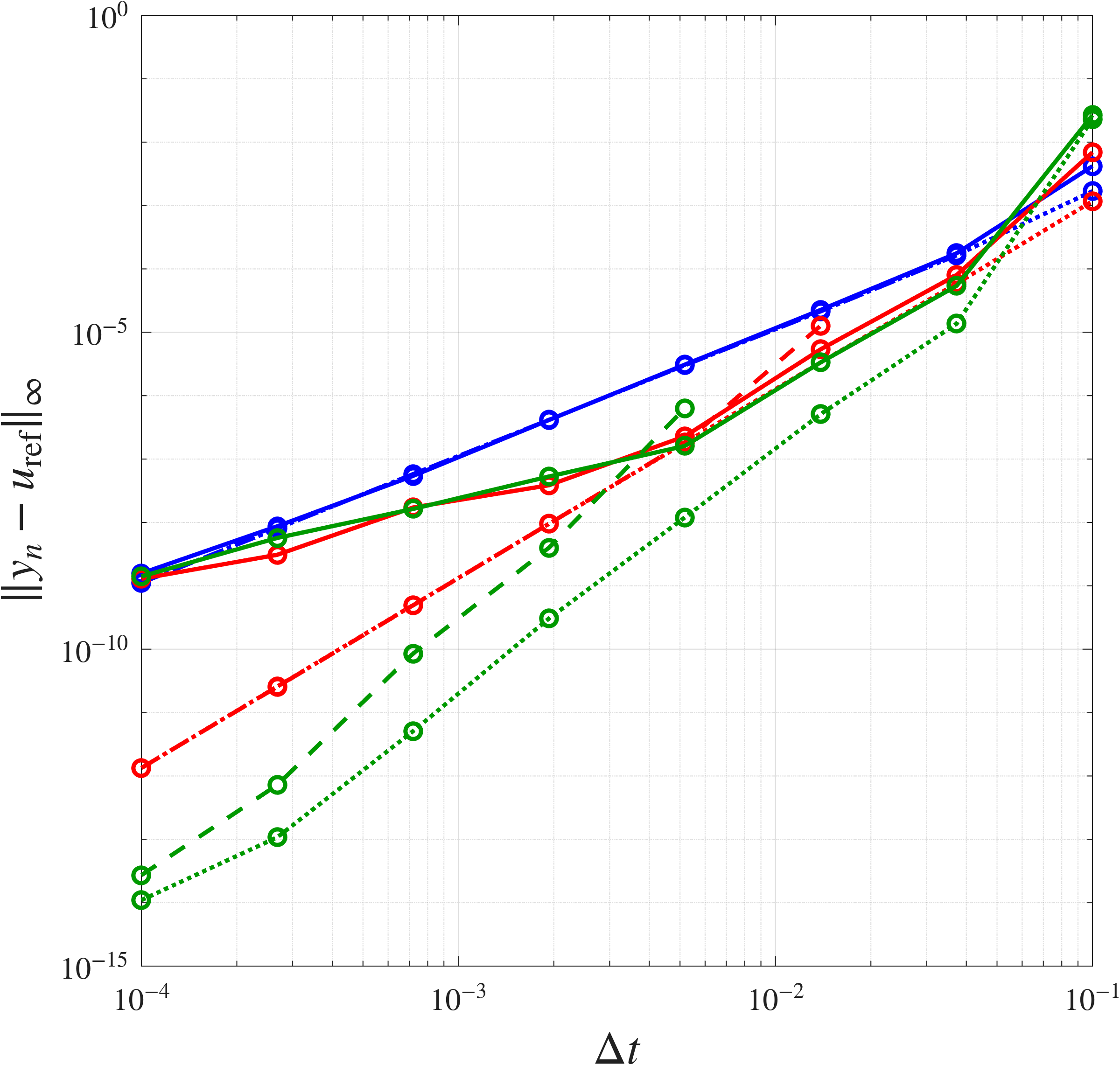}
\end{subfigure} \\
\caption{Linearized porous medium 
equations with $N_x =32$ points.
Blue, red and green are for second order SDIRK2, third order SDIRK3 
and fourth order SDIRK4 respectively.
No corrections (solid lines),
explicit correction (dashed line),
static Jacobian-based stabilized correction (dotted lines).
Left: one correction;  Right: $p-1$ corrections.
Top: no perturbations. Bottom: perturbation of $\epsilon = 10^{-4}$.
\label{fig:PMcorNx32}}
\end{figure}

A  Taylor series linearization is:
\begin{eqnarray} \label{eq:PMlinearize}
f_{\epsilon}(y) &=&  f(\bar{y}) + 
    f'(\bar{y}) \left( y - \bar{y} \right) 
    =   D_{xx} \bar{y}^3 +  3 D_{xx} \bar{Y}^2 \left( y - \bar{y} \right)  ,
\end{eqnarray}
where we linearize around  $\bar{y} = u^n$.
We  can additionally perturb the matrix for the implicit solve 
\begin{eqnarray} \label{eq:PMperturb}
(I - 3 a_{ii} \dt D_{xx} \bar{Y}^2)^{-1} + pert_{\epsilon}
\end{eqnarray}
where $pert_{\epsilon}$ represents truncating
each element of the matrix after a set number of digits $d$
leading to  a perturbation of $\epsilon = 10^{-d}$.

We tested the different correction strategies. The  static frozen Jacobian  
$\Phi_J = \big(I - 3 \mu \dt  D_{xx} \bar{Y_0}^2 \big)^{-1}$,
outperformed the static EIN stabilization $\Phi_{EIN} = \big(I - \mu \dt D_{xx} \big)^{-1} $
approach, as well as the dynamic $\Phi_B$ using ``Bad Broyden's'' algorithm. 
Once again, we let $\mu = a_{ii}$ in these simulations, where  $\mu =0$  
recovers the  explicit corrections.  The figures below show the uncorrected, 
explicit corrections, and the  static $\Phi_J$ stabilized corrections.

In Figure \ref{fig:PMcorNx32} we show the impact of corrections on 
the linearized porous medium equations with $Nx=32$
with no perturbations (top),
and with a perturbation of $\epsilon = 10^{-4}$ (bottom).
We compute the errors compared to a reference solution. Here, $\dt$
is refined but $N_x$ is constant.
In  blue, red and green are for the mixed accuracy SDIRK2, SDIRK3, and SDIRK4,
respectively. We see that in the absence of perturbations (top) and
without corrections (solid lines), the SDIRK2 is second order, 
and both the SDIRK3 and  SDIRK4 have third order errors. 
In the presence of a perturbation of four decimal places (bottom),
SDIRK2 is still second order, but the accuracy of the SDIRK3 and  SDIRK4 
degrades as $\dt $  gets smaller.

\begin{figure}[b!]
\centering
\begin{subfigure}[b]{\textwidth}
\centering
\includegraphics[width=0.48\textwidth]{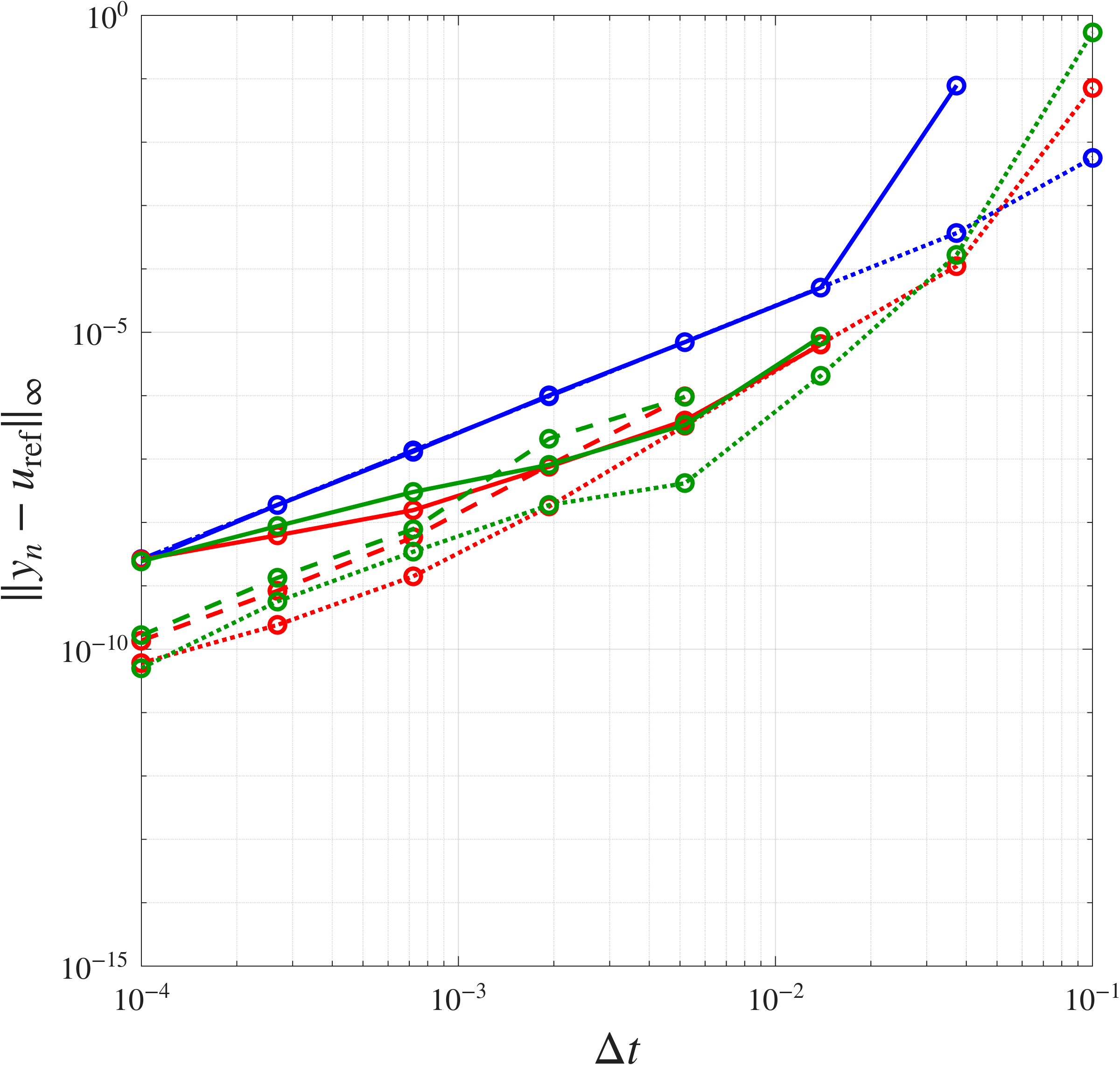}
\includegraphics[width=0.48\textwidth]{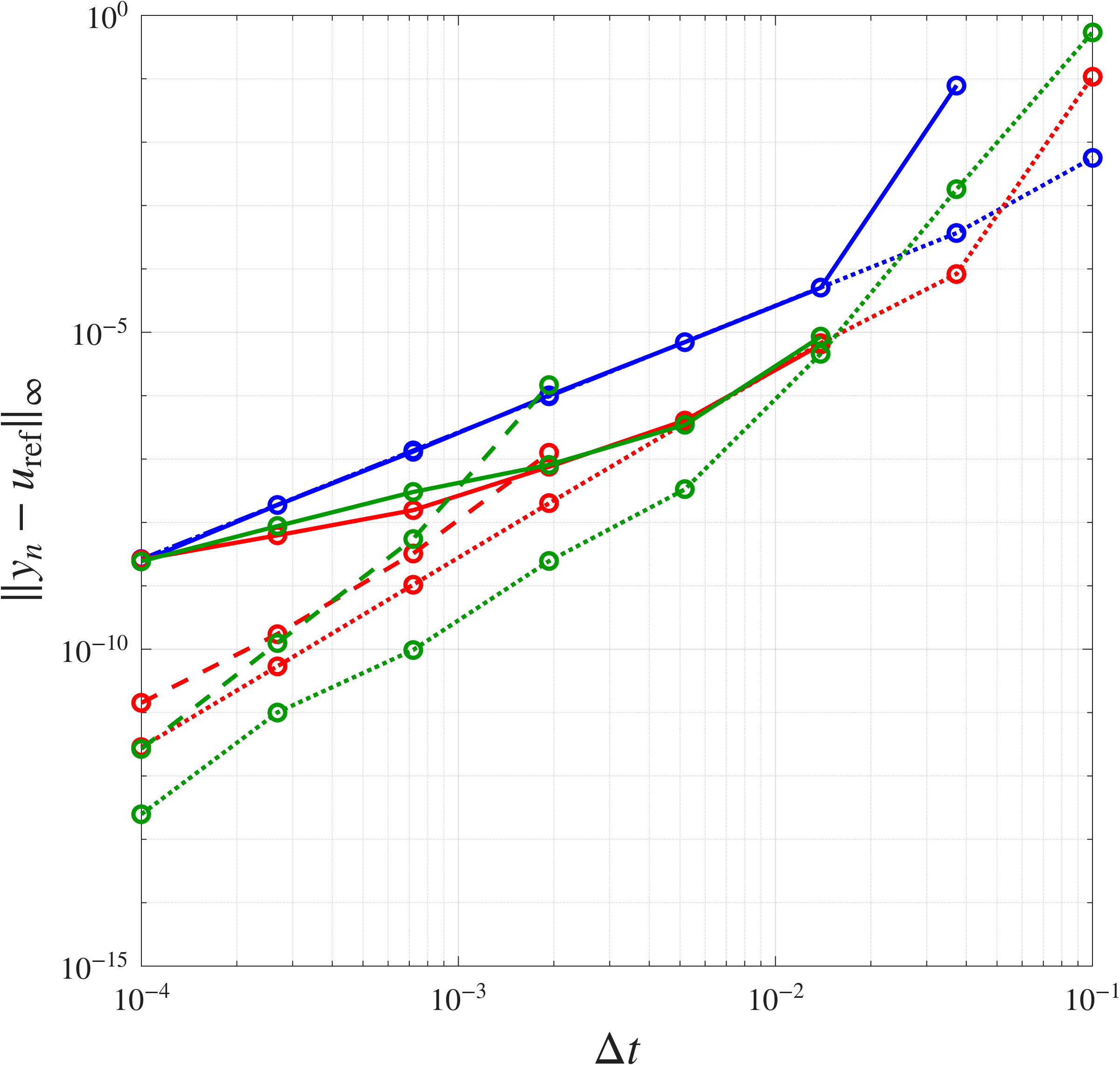}
\end{subfigure}
\caption{Linearized porous medium 
equations with $N_x =64$ points and a perturbation of $\epsilon =10^{-4}$.
Blue, red and green are for SDIRK2, SDIRK3 and SDIRK4 respectively.
Left: one correction. Right: $p-1$ corrections.
Solid lines are no corrections, 
dashed lines are explicit corrections,
dotted lines are $\Phi_J$ stabilized static corrections.
\label{fig:PMCorNx65ep4}}
\end{figure}

On the top of Figure \ref{fig:PMcorNx32}, 
we see that the  explicit corrections (dashed line) 
improve the accuracy of SDIRK3 and get the correct order for 
the SDIRK4 without  perturbations.
However, the corrected method becomes unstable when $\dt$ is sufficiently large 
(this is seen as the dashed lines disappear).
As more explicit corrections are added (top right) this instability appears sooner,
i.e. for a smaller $\dt$.
When the  stabilized corrections are added (using  $\Phi_J$)
we observe that the stability as well as accuracy is improved (dotted lines). 

On the bottom of Figure \ref{fig:PMcorNx32}, we repeat this process with a
perturbation of $\epsilon = 10^{-4}$. The same behavior is seen for the explicit
correction: some improvement in accuracy for small enough $\dt$, but for larger $\dt$
the explicit corrections cause instability.
Here we see that the stabilized corrections (dotted lines) improve both the stability
and accuracy of the method, this impact is most pronounced when we have $p-1$
corrections (bottom left).

In Figures \ref{fig:PMCorNx65ep4} we show the impact of the corrections
for a larger problem $N_x=64$, with a perturbations of $\epsilon = 10^{-4}$.
The explicit corrections (dashed lines) perform as we've come to expect: 
they improve accuracy but only for small enough $\dt$. For larger $\dt$
these explicit corrections may lead to catastrophic instabilities. What we see in this figure
that we have not seen before is that all the uncorrected methods (SDIRK2,
SDIRK3, and SDIRK4) are unstable for large enough $\dt$, and they 
are {\em stabilized and corrected} using the stabilized corrections
(dotted lines). In this example, the stabilized corrections not only 
improve the accuracy of the method, they improve the stability as well.

\begin{figure}[H]
     \centering
    \begin{subfigure}[b]{\textwidth}
        \centering
        \includegraphics[width=0.49\textwidth]{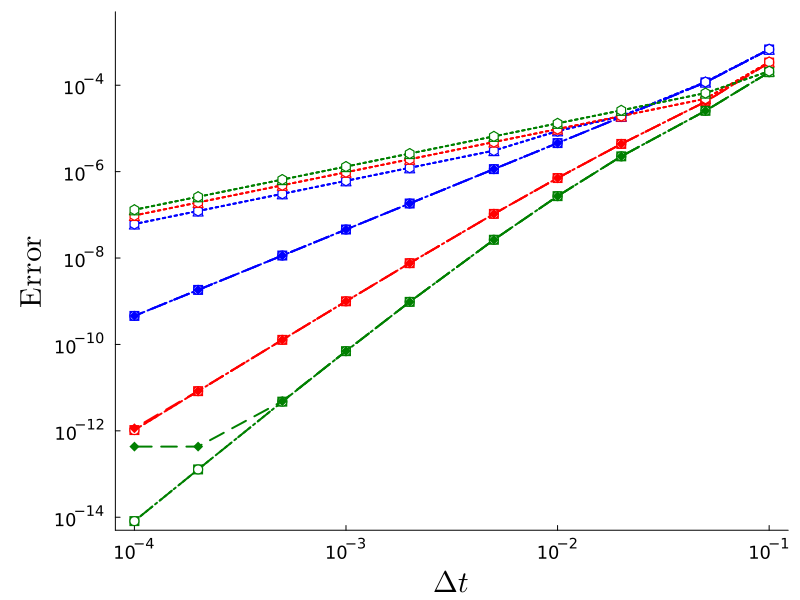}
        \includegraphics[width=0.49\textwidth]{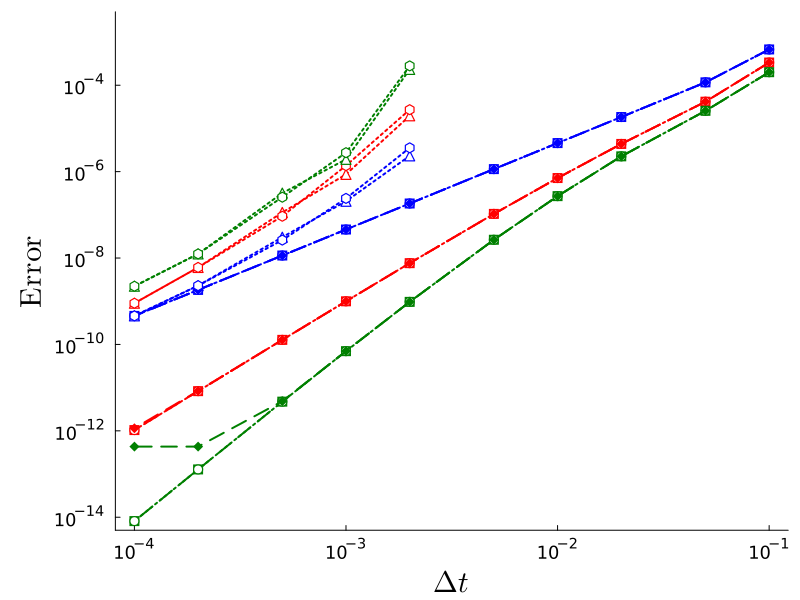}
    \end{subfigure} \\
    \begin{subfigure}[b]{\textwidth}
        \centering
        \includegraphics[width=0.49\textwidth]{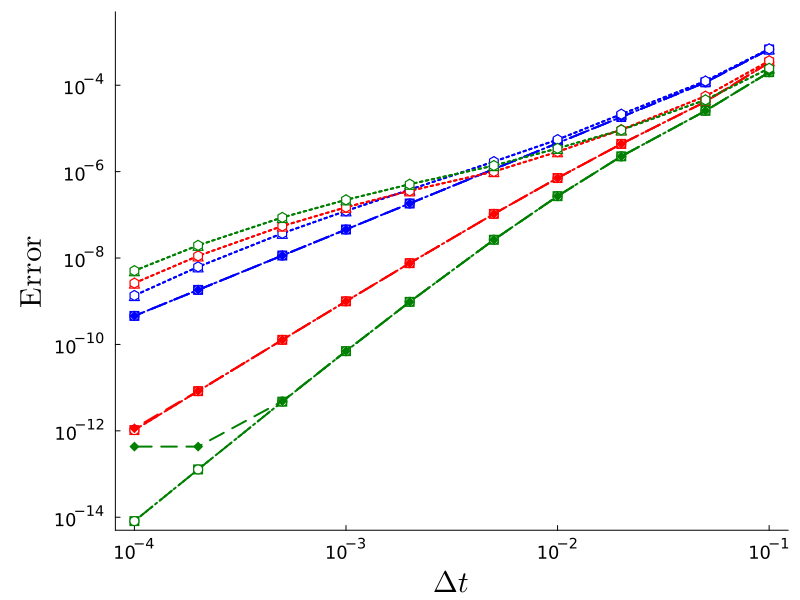}
        \includegraphics[width=0.49\textwidth]{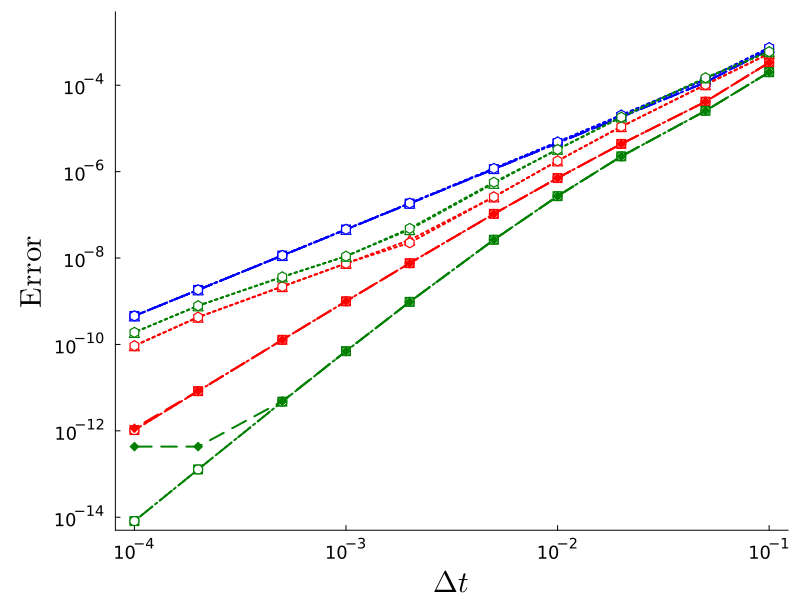}
    \end{subfigure}
\caption{Mixed precision porous medium equations with
$N_x = 200$ spatial points.  
SDIRK2 \eqref{MPIMR} in blue;  
SDIRK3 \eqref{MPSDIRK3} in red; 
SDIRK4 \eqref{MPSDIRK4} in green. 
Top Left: No corrections. 
Top Right: Two explicit corrections. 
Bottom Left:  Two $\Phi_{EIN}$ stabilized corrections.  
Bottom Right: Two $\Phi_J$ stabilized corrections.
(See legend in Burgers' mixed precision figure).
    \label{fig:PMMPFT}} 
\end{figure} \vspace{-.4in}

\subsubsection{Mixed precision implementation}
Here we use Equation \eqref{eq:PM} on domain $(0,2 \pi)$ and initial
condition $u(x,0)= \frac{1}{2} \sin(x)$.
We compute the implicit solve in low precision as described above.
This procedure includes an inherent correction which allows an accuracy of $\varepsilon \dt$
at the final time. To further correct, we can use the explicit 
correction in high precision
\[ y_{[k+1]}^e =y_{exp} + \alpha \dt D_{xx}( y_{[k]}^3) .\]
The stabilized corrections are then applied in high precision
 \[ y_{[k+1]} = y_{[k]} + \Phi  \left(y_{[k+1]}^e - y_{[k]} \right) ,\]
where  the stabilization matrix $\Phi$ is computed in high precision
\[\Phi = \big(I - \mu \dt D_{xx} \big)^{-1} .\]

In Figure \ref{fig:PMMPFT} we show the impact of the mixed precision procedure
on the errors of the method, compared to a reference solution with 
$N_x = 200$ spatial points.  We show the
SDIRK2 \eqref{MPIMR} in blue, the 
SDIRK3 \eqref{MPSDIRK3} in red, and the 
SDIRK4 \eqref{MPSDIRK4} in green.
On the top left we see that in the absence of corrections
the mixed precision simulations have similar poor performance: the perturbation errors
dominate the solutions. Two explicit corrections (top right) improve the accuracy for very
small $\dt$ but ruin the stability for slightly larger $\dt$ (note the dotted lines disappearing).
 Two $\Phi_{EIN}$ stabilized corrections remain stable and correct the errors,
 but two $\Phi_J$ stabilized corrections are even more effective at improving the accuracy.
 
Figure \ref{fig:PMMPFrozen} investigates further the effect of 
$\Phi_J$ stabilized corrections from one to three corrections.
We see that more $\Phi_J$ corrections continually improve the accuracy of the solution
without adversely impacting stability for smaller $\dt$. This is less
clear-cut as $\dt$ is larger, in which case fewer corrections may be better. 
This result implies that more corrections may continue to provide improvement 
for some values of $\dt$, but not others. This   suggests that a strategy 
that measures the residual and sets a  tolerance for correction as well as a maximum
number of corrections may be advantageous in practice.

\begin{figure}[!htb]
     \centering
    \begin{subfigure}[b]{\textwidth}
        \centering
        \includegraphics[width=0.32\textwidth]{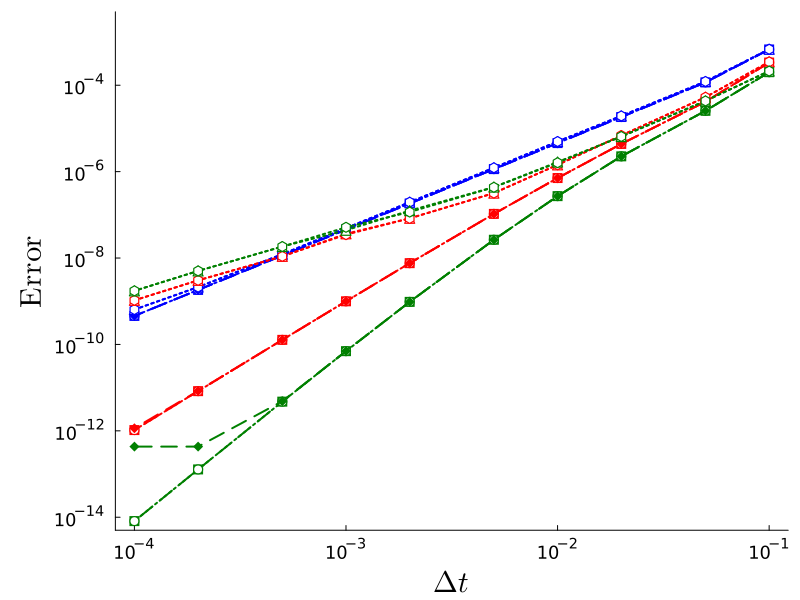}
        \includegraphics[width=0.32\textwidth]{Plots/MP_PM/Figure_AllMethods_Quad_Nx200__2_corr_1.png}
        \includegraphics[width=0.32\textwidth]{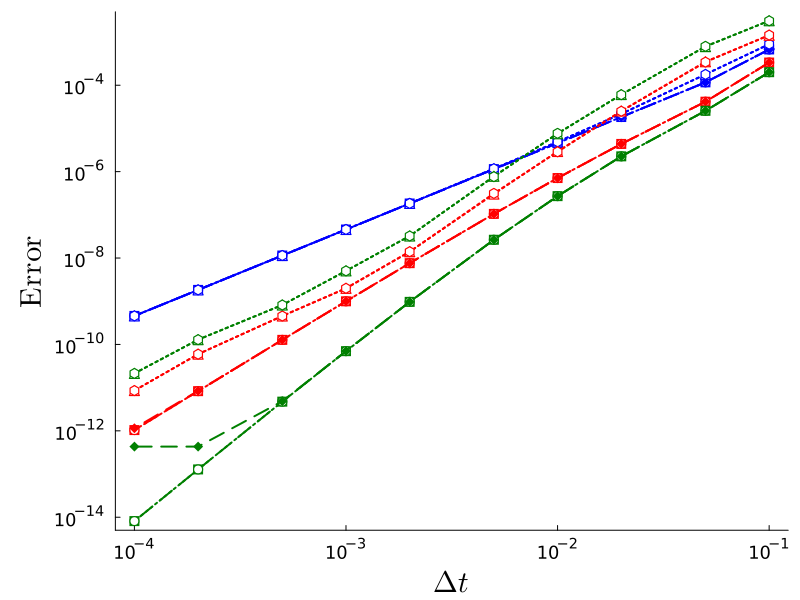}
    \end{subfigure} 
\caption{Mixed precision porous medium equations with
$N_x = 200$ spatial points.  
SDIRK2 \eqref{MPIMR} in blue;  
SDIRK3 method \eqref{MPSDIRK3} in red; 
SDIRK4 method \eqref{MPSDIRK4} in green. 
Left:  one $\Phi_J $ correction. 
Middle:  two $\Phi_J $ corrections.  
Right: three  $\Phi_J $ correction.
    \label{fig:PMMPFrozen}} \vspace{-.2in}
\end{figure}

% \begin{figure}[t]
% \centering
% \begin{subfigure}[b]{\textwidth}
% \centering
% \includegraphics[width=0.48\textwidth]{Plots/MP_PM/MP_PMNx250Cor.jpeg}
% \includegraphics[width=0.48\textwidth]{Plots/MP_PM/MP_PMNx252ExCor.jpeg}
% \end{subfigure} \\
% \begin{subfigure}[b]{\textwidth}
% \centering
% \includegraphics[width=0.48\textwidth]{Plots/MP_PM/MP_PMNx252EINCor.jpeg}
% \includegraphics[width=0.48\textwidth]{Plots/MP_PM/MP_PMNx252FJCor.jpeg}
% \end{subfigure} \\
% \caption{Linearized porous medium 
% equations with $N_x =24$ points.
% Blue, red and green are for second order SDIRK2, third order SDIRK3 
% and fourth order SDIRK4 respectively.
% Top left: no corrections;
% Top right: explicit correction;
% Bottom left: static EIN-based stabilized correction;
% Bottom right: static Jacobian-based stabilized correction.
% \label{fig:MPPMcorNx24}}
% \end{figure}

\section{Conclusions} \label{sec:conclusions}

In this work we analyzed the impact of the perturbation errors  
of mixed accuracy DIRK methods \eqref{MPDIRK}
with coefficients that satisfy the conditions \eqref{Coefficients}.
We showed that for contractive problems, the perturbation introduced 
by replacing $f$ with $f_\varepsilon$ in the implicit solve results 
(for large enough $\dt$) in  an error growth at each time-step of
\[ \dt^2 L \Theta \]
where $\Theta = O(\varepsilon)$.
Thus we can conclude that the error growth at some fixed final time $T_f$
is bounded by  
\[ O(\varepsilon \dt L  T_f) .\]
%{\color{red}  Here I think we should just refer to the equation where theta is defined rather than repeat the definition of Theta.} -- Done.
This means that the errors only grow linearly with time, and that provided that 
the perturbation $ \varepsilon $ is small enough compared to the 
time-step and stiffness of the problem, the error growth over time is well-behaved.

We note that the accuracy conditions were described in \cite{Grant2022}. In that work it was shown
why a  low precision implementation will lead to a growth of $O(\varepsilon/\dt)$ over time,
while a naive mixed precision implementation (where all $f$ are in low precision but the rest 
of the solution is in high precision) will lead to a growth or  $O(\varepsilon)$ over time.
In that work, the order conditions were described that allow this type of $O(\varepsilon \dt)$ 
growth over time. In this work we build on this result by providing the stability 
analysis that tracks the growth of the errors over time 
and allows us to understand how to control the final time error
by controlling the perturbation $\varepsilon$ and the design of the method.
Unfortunately, we cannot always directly control  $\varepsilon$, which is determined by 
the type of approximation $f_\epsilon$, which in turn may depend on $\dt$,  machine precision 
$\epsilon_{prec}$, and size of the  system $N_x$.  In addition, $\varepsilon$ itself 
may inherit some of the stiffness $L$ of the problem.

To better damp out these perturbation errors and  improve the order of accuracy of the perturbed method, explicit corrections were proposed in 
\cite{Grant2022} and studied in \cite{Burnett1, Burnett2}. 
While these do an excellent job improving the accuracy of the solution for small enough $\dt$, 
they may adversely impact  the stability of the numerical solution when $\dt$ is large.
In this work, we propose a strategy for stabilizing these corrections, 
and describe several choices for the stabilization matrix.  Using the analysis presented in 
Section \eqref{sec:stability} we can explain how these corrections improve the stability 
and accuracy of the solution. We  also numerically explore the stability and accuracy of 
the stabilized correction approach on three test cases. This analysis opens the possibility
of exploring inexpensive and stable corrections that allow us to efficiently
implement mixed accuracy and mixed precision problems while obtaining highly accurate solutions.

\section*{\sc Acknowledgements}
This material is based upon work supported by the 
National Science Foundation under Grant No. DMS-1929284 
while four of the  authors were in residence at the Institute for Computational and Experimental Research in Mathematics in Providence, RI, during the ``Empowering a Diverse Computational Mathematics Research Community'' program.
The authors' research was supported in part by AFOSR Grant No. FA9550-23-1-0037 and
DOE  Grant No. DE-SC0023164 Subaward RC114586.
SG acknowledges the support of  Mass Dartmouth’s Marine and Undersea Technology 
(MUST) Research Program funded by the ONR Grant No. N00014-20-1-2849. 
MS acknowledges support from the National Science Foundation PRIMES program under Grant No. DMS-2331890.
The authors acknowledge the Unity Cluster managed by the Research Computing 
\& Data team at the University of Massachusetts Amherst, and the UMassD shared 
cluster as part of the Unity cluster, supported by  
AFOSR DURIP grant FA9550-22-1-0107.

\section*{\sc Author Contribution Statement}

\noindent{\bf  John Driscoll} investigated numerous approaches and test cases.
He was primarily responsible for coding up the EIN and Jacobian-base approaches for
stabilizing the linearizations and perturbation. JD reviewed the entire manuscript and 
suggested edits.

\noindent{\bf Sigal Gottlieb} was responsible for conceptualization of this project, 
and worked as part of the original research team at the ICERM summer program.
With ZJG, She was primarily responsible for much of the analysis, 
including that in Theorem \ref{thm:FinalConv}. 
She suggested numerical tests and the mixed precision correction approaches.
SG was primarily responsible for writing and editing the manuscript.

\noindent{\bf Zachary J. Grant} worked as part of the original 
research team at the ICERM summer program and was part of the project 
from shortly after the conceptualization. He was originally responsible 
for developing the perturbation  framework for DIRK methods, for the explicit 
corrections, and for  the idea of stabilizing the explicit corrections with an additional 
implicit term. With SG, he was primarily responsible for much of the analysis, especially 
the matrix based correction analysis. ZJG co-wrote Sections 2, 3, 4, and 5.
He carefully proofread the entire paper,  and made numerous editorial suggestions 
to improve the presentation.

\noindent{\bf César Herrera}
worked as part of the original research team at the ICERM summer program
and was part of the project from shortly after the conceptualization.
He was  involved in discussions on the underlying ideas for efficient and stable corrections.
CH was responsible for many numerical tests and graphs.
He provided the expertise on mixed precision implementation in julia language, 
and all the related numerical results. He read, commented, and edited the entire manuscript.

\noindent{\bf Tej Sai Kakumanu} was primarily responsible for the Broyden correction
approaches. He investigated many Broyden-based strategies for stabilized corrections,
including the per-step, per-stage, and per-iteration approaches.
He was primarily responsible for coding up the Broyden-based stabilized 
corrections for the linearized and perturbed problems.
TSK reviewed the entire manuscript and  suggested edits.

\noindent{\bf Michael H. Sawicki}
studied Jacobian and Broyden based stabilized corrections. He 
wrote code for many linearization and correction based simulations.
MHS showed that Broyden corrections of an initial Jacobian based $\Phi$
perform better than continual corrections.

\noindent{\bf Monica Stephens} 
worked as part of the original research team at the ICERM summer program
and was part of the project from shortly after the conceptualization.
She tested multiple correction approaches based on the explicit methods.
MS carefully reviewed the entire manuscript for  correctness
and  made multiple editorial changes that contributed to the clarity of the 
manuscript.

\section*{\sc Availability of Codes}
{\em All codes will be place in a github repository before publication. The codes are currently being
cleaned up, commented, and made easier for people to run. The final version of this manuscript will
list the code information and github addess in this section.
}

\end{document}